\documentclass[11pt]{article}
\usepackage[utf8]{inputenc}
\usepackage{amsfonts} 
\usepackage{amssymb}
\usepackage{amsmath,amsthm}
\usepackage{graphicx}
\usepackage{xcolor}
\usepackage{tikz-cd}
\usepackage{comment}
\usepackage{txfonts}
\usepackage{lipsum}
\usepackage{amscd}
\usepackage{ulem}
\usepackage[letterpaper, left=2.1cm, right=2cm, top=3.5cm, bottom=3.5cm]{geometry}

\newcommand\blfootnote[1]{%
  \begingroup
  \renewcommand\thefootnote{}\footnote{#1}%
  \addtocounter{footnote}{-1}%
  \endgroup
}
\providecommand{\U}[1]{\protect\rule{.1in}{.1in}}

\makeatletter
\newtheorem*{rep@theorem}{\rep@title}
\newcommand{\newreptheorem}[2]{%
\newenvironment{rep#1}[1]{%
 \def\rep@title{#2 \ref{##1}}%
 \begin{rep@theorem}}%
 {\end{rep@theorem}}}
\makeatother

\makeatletter
\newtheorem*{rep@corollary}{\rep@title}
\newcommand{\newrepcorollary}[2]{%
\newenvironment{rep#1}[1]{%
 \def\rep@title{#2 \ref{##1}}%
 \begin{rep@corollary}}%
 {\end{rep@corollary}}}
\makeatother

\newtheorem{theorem}{Theorem}[section]
\newreptheorem{Theorem}{Theorem}
\newrepcorollary{Corollary}{Corollary}

\newtheorem{corollary}[theorem]{Corollary}

\newtheorem{definition}[theorem]{Definition}

\newtheorem{lemma}[theorem]{Lemma}
\newtheorem{question}[theorem]{Question}

\newtheorem{proposition}[theorem]{Proposition}
\newtheorem{remark}[theorem]{Remark}

\newcommand{\N}{\mathbb{N}} 
\newcommand{\K}{\mathcal{K}}

\newcommand{\C}{\mathcal{C}}

\newcommand{\Kcal}{\mathcal{K}}

\newcommand{\cQ}{\mathcal{Q}}

\title{Local entropy theory and descriptive complexity}
\author{Udayan B. Darji and Felipe García-Ramos}
\date{}

\begin{document}

\maketitle
\abstract{
We investigate local entropy theory, particularly the property of having completely positive entropy (CPE), from a descriptive set-theoretic point of view. We aim to determine descriptive complexity of different families of dynamical systems with CPE. 

}

For a large class of compact $X$, we show that the family of dynamical systems on $X$ with CPE is complete coanalytic and hence not Borel. When we restrict our attention to dynamical systems having  special properties such as the mixing property or the shadowing property, we obtain some contrasting behavior. In particular, the notion of CPE and the notion of uniform positive entropy, a Borel property,
coincide for mixing maps on topological graphs. On the other hand, the class of mixing map on the Cantor space is coanalytic and not Borel. 

For dynamical systems with the shadowing property, the notions CPE and uniform positive entropy coincide regardless of the phase space.

\section{Introduction}
Take two families of dynamical systems; how can we formally compare  the dynamical diversity/complexity of the chaotic behavior of each family? In this paper we study the chaotic components of a system in the framework of local entropy theory, and we use descriptive set theory to formally compare their complexity.

Local entropy theory is a culmination of deep results in combinatorics and topological dynamics. It gives us a way to understand ``where" the entropy lies within a system and helps us  recognize combinatorially
the emergence of entropy using the notion of independence. We refer the reader to the survey article by
Glasner and Ye \cite{glasner2009local} and the book by Kerr and Li \cite[Chapter 12]{kerr2016ergodic} for more information on the subject.  It turns out that local entropy theory is a powerful tool that can be applied in a variety of settings. For example, local entropy theory provides a direct and an intuitive reason as to why positive entropy is stronger than Li-Yorke chaos \cite{blanchard2002li,KerrLiMA}.  It can also be used for finding homoclinic points of algebraic actions \cite{chungLi2015,barbieri2019markovian}. In another application of this theory, the first author and Kato \cite{darji2017chaos} settled some
old problems concerning indecomposable continua in dynamical systems.

Local entropy theory was introduced by Blanchard \cite{blanchard93}. His original motivation was to understand the topological analogues of $K$-systems, i.e., ergodic systems  whose non-trivial factors have positive entropy.  To this end, Blanchard gave two possible interpretations: uniform positive entropy (UPE) and complete positive entropy (CPE). Blanchard realized that both of these conditions can be characterized using a new local concept. Thus was born the notion of \emph{entropy pairs} and the area of local entropy theory. This theory is useful, not only for understanding positive entropy, but also for understanding zero entropy systems such as null and tame systems. Building up on seminal work of various mathematicians and developing new elegant ideas, Huang and Ye \cite{huang2006local}, and Kerr and Li \cite{KerrLiMA} provided a unified approach to local entropy theory. This unified approach is of combinatorial nature and uses the concept of independence.   

In his original article \cite{Blanchard1992}, Blanchard showed that every topological dynamical system (TDS) that has UPE also has CPE but the converse is not true. Using a natural rank called the  $\Gamma$-rank, Barbieri and the second author \cite{barbieri2020} constructed a transfinite hierarchy of families that have CPE but not UPE. This was the first inkling of the relationship between local entropy theory and descriptive set theory. 

Topological dynamics and ergodic theory are sibling branches of dynamical systems in which several notions can be translated from one branch to the other (for surveys on these connections see \cite{glasner2006interplay,huangsurvey}).
For a fixed compact metrizable space $X$, respectively a probability space $X$, we can consider the space of all topological dynamical systems, respectively all measure-preserving systems, on $X$. Equipped with a suitable natural topology, such collections form Polish spaces. Now one can ask several interesting questions. For example, how big is a subcollection of dynamical systems with a particular property? or what is its descriptive complexity? In particular, is it Borel? 

The study of such questions  has a long history. In the context of ergodic theory,  we have a result of Halmos  of 1944 \cite{halmos1944general} which states that among measure preserving transformations those that are  weak-mixing forms a residual $G_{\delta}$ set.
The first ``anti-classification" or non-Borel result in dynamics emerged in a paper by Beleznay and Foreman \cite{beleznay1995collection} where they proved that the collection of distal topological dynamical systems is complete coanalytic and hence not Borel.

From the descriptive complexity viewpoint, most of the classical families of measurable chaotic dynamical systems are Borel: mixing \cite{rohlin1948general}, Bernoulli \cite{feldman1974borel}, positive entropy and K-systems \cite{foreman2000descriptive}. In the context of topological dynamics, it is fairly routine to show that families of systems which are 
weak-mixing, mixing, Bernoulli shifts and have  positive entropy are all Borel. Now, as exhibited by Blanchard, K-systems have at least two distinct interpretations in topological dynamics, namely UPE and CPE. It turns out that UPE again is Borel, Corollary~\ref{cor:upeborel}. Nonetheless, it follows from results of  Barbieri and the second author \cite{barbieri2020}, and Westrick \cite{westrick2019topological}, that CPE is not Borel. Informally speaking, this implies that no amount of inherently countable resources can settle the question of whether a TDS has CPE or not.



 There are also some recent results regarding effective (or light face) descriptive set theory, CPE, and symbolic dynamics.
The collection of SFTs is a countable set and hence $F_{\sigma}$. Thus, effective descriptive set theory is the appropriate context for determining their complexity. It follows from a result of Pavlov \cite{pavlov2018topologically}, that SFTs with zero-dimensional CPE (factors having dimension zero) is effectively Borel. Meanwhile, Westrick \cite{westrick2019topological} showed that the family of $\mathbb{Z}^{2}$-subshifts of finite type (SFT) having CPE is effectively complete coanalytic. Returning to classical descriptive set theory,
Salo \cite{salo2019entropy} constructed subshifts on ${\mathbb Z}$ that have CPE and arbitrarily high entropy rank; this result can be used to show that among the space of all ${\mathbb Z}$-subshifts, the ones having CPE form a non-Borel collection.


In this article we initiate a systematic study of relationship between local entropy theory, in particular properties of UPE and CPE, and descriptive complexity. As mentioned earlier, the collection of systems on a Cantor space having CPE is coanalytic and not Borel \cite{barbieri2019markovian}, \cite{westrick2019topological} and this can be further refined by making the maps shift maps \cite{salo2019entropy}, \cite{westrick2019topological}. At this point two questions present themselves: does the descriptive complexity change as we vary the underlying phase space and as we require extra conditions on the map.

Our first result is that for a very general class of phase spaces, if the maps are left unrestricted, then the systems having CPE form a complete coanalytic set. (Recall that all complete coanalytic set are non-Borel.) Below we state it for the interval in ${\mathbb R}^n$ and in Remark~\ref{rem:topmanifolds}, we explain how to obtain the result for graphs and orientable manifolds.

\begin{repTheorem}{intcomplete}
For every $d\geq 1$ we have that $\textup{CPE}(I^d)$ is complete coanalytic.
\end{repTheorem}

A curious phenomenon occurs when we require the map to have finer properties. A result of Blokh \cite{blokh1984transitive} implies that for the interval and, in general, any topological graph $X$, every mixing system having CPE  also has UPE. As the collection of UPE is Borel, we have the following corollary.
\begin{repCorollary}{cor:graphBorel}
Let $X$ be a topological graph. Then $\textup{Mix}(X)\cap \textup{CPE}(X)$ is Borel.
\end{repCorollary}
A similar situation occurs for maps with the shadowing property, that is, the properties of UPE and CPE coincide, Corollary~\ref{upeshadow}, however, this time on any compact metrizable space. As such we have the following corollary.
\begin{repCorollary}{shadowcpeborel}
Let $X$ be a compact metrizable space. Then $\textup{Shad}(X)\cap \textup{CPE}(X)$ is Borel. 
\end{repCorollary}
It may seem that for maps which are of most interest to dynamicists such as transitive, mixing, chaotic, etc, the property of having CPE is Borel.
The main result of this article is to show that this is false. Namely, the following holds.
\begin{repTheorem}{mainmixincpe}
Let $X$ be a Cantor space. Then, $\textup{Mix}(X) \cap \textup{CPE}(X)$ is  coanalytic and not Borel. 
\end{repTheorem}





The proof of Theorem~\ref{intcomplete} is short and somewhat easy to verify once the main idea is stated.
On the other hand, the proof of Theorem~\ref{mainmixincpe} is rather delicate and involved, {as evident in Section~\ref{topcon}}.  It is not surprising that proving that a family of dynamical systems is not Borel is technically  difficult for transitive and mixing maps. Usually, the easiest path to proving that a family of non-reducible systems (i.e. non-ergodic or non-transitive) is not Borel, is to build examples with countably many components (as done in \cite{barbieri2020} and in Section 3). This is analogous to what happened for the conjugacy problem in ergodic theory. First, Hjorth proved that the equivalence relation generated by the conjugacy of measure preserving transformation is not Borel \cite{hjorth2001invariants}. Later, Foreman, Rudolph and Weiss \cite{foreman2011conjugacy}, using a much more delicate construction, proved the same  for ergodic systems.




The paper is organized as follows. In Section 2 we give a somewhat detailed introduction to local entropy theory and descriptive set theory so that the article is accessible. In Section 3 we investigate relationship between shadowing and local entropy. In Section 4 we study the complexity of systems with CPE on  topological manifolds. In Section 5 we prove the main result of the paper. Several open questions arise from the above which will be discussed in Section 6. 


\textbf{Acknowledgment:} The authors would like to thank Dominik Kwietniak,
Slawomir Solecki and Linda Westrick for motivating conversations. 
The second author was supported by the CONACyT grant 287764.

\section{Preliminaries}
In this section we will give the basic definitions needed for the paper, as well as an introduction to local entropy theory and descriptive set theory for non-experts. 

Let $Z$ be a set and $F\subseteq Z^2$. We define the transitive closure of $F$ as follows
\[
F^{+}=\{(u,v)\in Z^2:\exists u=u_1,...,u_n=v \text{ s.t. } (u_i,u_{i+1})\in F\}\text{, and}
\] 
\[
\Delta_Z=\{(u,u):u\in Z\}.
\]
A set $F\subseteq Z^{2}$ is an equivalence relation if and only if $F$ is symmetric,
$\Delta_{Z}\subseteq F$, and $F=F^{+}$. In this case we denote the equivalence
class of $z\in Z$ with $[z]_{F}$.

Throughout this paper, $X$ is always a compact metrizable space with a compatible metric $d$. 

Let $E\subseteq X^{2}$ be a symmetric set. We define
\[
\Gamma(E)=\overline{E^{+}\cup \Delta_X}.
\]
For an ordinal $\alpha$, $\Gamma^{\alpha}(E)$ is defined by 
\[
\Gamma^{\alpha}(E) = \Gamma( \Gamma^{\alpha -1}(E )),
\]
if $\alpha$ is the successor ordinal and 
\[
\Gamma^{\alpha}(E) = \overline {\cup_{\beta < \alpha} \Gamma^{\beta}(E)},
\]
if $\alpha$ is a limit ordinal.  

{As every compact metric space has a countable basis, we have that every strictly increasing transfinite sequence of closed sets must be countable.} 

From this we have the following.
\begin{proposition}
Let $X$ be a compact metrizable space and $E\subseteq X^2$ symmetric.
There exists a countable ordinal $\alpha$ such that $\Gamma^{\alpha}%
(E)=\Gamma^{\alpha+1}(E)$.
\end{proposition}

 The smallest such ordinal is called the $\Gamma$-\textbf{rank of} $E$. We define $\Gamma^{\infty} (E)=\Gamma^{\alpha} (E)$, where $\alpha$ is the rank of $E$.
 
 The \textbf{Cantor set} $C\subseteq [0,1]$ is the standard Cantor ternary set. A \textbf{Cantor space} is a non-empty compact metrizable space that is totally disconnected (i.e., the only connected subsets are singletons) and has no isolated points. Every Cantor space is homeomorphic to the Cantor set \cite[Corollary 30.4]{willard}.
 
\subsection{Topological dynamics and local entropy theory}
We say that $(X,T)$ is a \textbf{topological dynamical system (TDS)} if $X$ is a compact metrizable space and $T:X\rightarrow X$ is a continuous function. A TDS is (topologically) \textbf{mixing} if for every pair of nonempty open sets $U,V$ there exists $N>0$ such that $T^{-n}(U)\cap V\neq \emptyset$ for every $n\geq N$. We say $(X_2,T_2)$ is a \textbf{factor} of $(X_1,T_1)$ if there exists a surjective continuous function $\phi : X_1\rightarrow X_2$ (called a \textbf{factor map}) such that $\varphi\circ T_{1}=T_{2}\circ\varphi$. It is well-known that all factors of mixing TDSs  and all countable products of mixing TDSs are mixing. 

Let $(X,T)$ be a TDS and $\mathcal{U},\mathcal{V}$ open covers of $X$. We denote the smallest cardinality of a subcover of $\mathcal{U}$ with $N(\mathcal{U})$, and
\[
\mathcal{U}\vee\mathcal{V}=\{U\cap V:U\in\mathcal{U}\text{ and }V\in\mathcal{V}\}.
\]
We define the \textbf{entropy of
$(X,T)$ with respect to $\mathcal{U}$} as
\[
h_{\text{top}}(X,T,\mathcal{U})=\lim_{n\rightarrow\infty}\frac{1}{n }\log N(\vee^n_{m=1}T^{-m}(\mathcal{U})).
\]
The \textbf{(topological) entropy} of $(X,T)$ is defined as \[
h_{\text{top}}(X,T)=\sup_{\mathcal{U}}h_{\text{top}}(X,T,\mathcal{U}).
\]
 
In order for a TDS to have positive entropy it is only necessary for a small portion of the system to be chaotic. In a broad sense, local entropy theory tackles the following questions: where can we localize the entropy of a given system? how can we define a system that has entropy everywhere? 

For measure-preserving systems, the idea of entropy ``everywhere" goes back to Kolmogorov. A measure-preserving system is a $K$-system if it satisfies Kolmogorov's zero-one law. The Rohlin-Sinai theorem, states that a measure-preserving system is a $K$-system if and only if every non-trivial (measurable) factor has positive (measurable) entropy if and only if the system has positive (measurable) entropy with respect to any non-trivial partition \cite[Theorem 18.9]{glasner2003ergodic} (note that measurable entropy is defined using partitions and not open covers). Dynamicists often study analogous properties in the topological and measurable contexts (see the surveys \cite{glasner2006interplay,huangsurvey}). 
In \cite{Blanchard1992}, two possible topological analogues of the $K$-systems were introduced. 

\begin{definition}
A TDS has \textbf{completely positive entropy (CPE)} if every non-trivial factor has positive entropy. 

A TDS has \textbf{uniform positive entropy (UPE)} if for every open cover $\mathcal{U}$ consisting of non-dense sets, we have that $h_{\text{top}}(X,T,\mathcal{U})>0$. 
\end{definition}
It turns out that these two notions are not equivalent. Every system that has UPE has CPE but the converse does not hold. A TDS with UPE is not necessarily mixing, but it is always weak-mixing \cite{Blanchard1992}. On the other hand there exist TDSs with CPE that are not transitive. 

The question: ``where can we localize the entropy?" has to do with how points relate to each other through the dynamics. Hence, this question is not answered in $X$ but in $X^2$. Informally, a pair is an entropy pair if every cover that separates the points has positive entropy. In this paper we will not give the original definition of Blanchard. Instead we will use the equivalent definition of IE-pairs (independence entropy pairs). The concept of independence (or interpolation) on dynamical systems appeared in the work Glasner and Weiss \cite{glasner1995quasi} (with ideas based on Rosenthal \cite{rosenthal1974characterization}), and Huang and Ye \cite{huang2004topological,huang2006local}. We will follow the definition of independence entropy pairs of Kerr and Li \cite{KerrLiMA}. 
 A set $I \subseteq \N$ has \textbf{positive density} if
$\liminf_n \frac{|I \cap [1,n]|}{n>0} >0$. Given a TDS $(X,T)$ and $\{U,V\}\subseteq X$, we say $I \subseteq \N$ is an \textbf{independence set for $\{U,V\}$} if for all finite $J \subseteq I$,
and for all $(Y_j) \in \prod _{j\in J}\{U,V\}$, we have that
\[
\cap_{j\in J}T^{-j}(Y_j)\neq \emptyset.
\]

\begin{definition}
Let $(X,T)$ be a TDS. We say that  $(x_1, x_2) \in X\times X$ is an \textbf{independence entropy pair (IE-pair) of $(X,T)$} if for every pair of open sets $A_1,A_2$, with $x_1\in A_1$ and $x_2\in A_2$, there exists an independence set for $\{A_1,
A_2\}$ with positive density. The set of IE-pairs of $(X,T)$ will be denoted by $E(X,T)$.
\end{definition}

We have that $E(X,T)\subseteq X^2$ is a closed subset. For the proof of the following results see \cite[Theorem 12.19]{kerr2016ergodic} and \cite{huang2007relative}. 
\begin{theorem}
\label{thm:basic}
Let $(X_1,T_1)$ and $(X_{2},T_{2})$ be TDSs. 

\begin{enumerate}
	
	\item $(X_1,T_1)$, has positive entropy if and only if there exists $x\neq y\in X_1$ with $(x,y)\in E(X_1,T_1)$.
	
	\item $E(X_1\times X_{2},T_1\times T_{2})=E(X_1,T_1)\times E(X_{2},T_{2})$. 
	
	\item Let $\phi:X_1\rightarrow X_2$ be a factor map. Then $E(X_2,T_2)=(\phi\times \phi)(E(X_1,T_1))$.
	
	\item $(X_1,T_1)$ and $(X_{2},T_{2})$ have CPE if and only if  $(X_1\times X_{2},T_1\times
	T_{2})$ has CPE.
	
	
\end{enumerate}	
\end{theorem}




\bigskip

Now we will study the $\Gamma$-rank defined at the beginning of this section in the particular case when $E=E(X,T)$ is the set of independence
entropy pairs.

\begin{definition}
Let $(X,T)$ be a TDS. The
\textbf{entropy rank} of $(X,T)$ is the $\Gamma$-rank of $E(X,T)$.
\end{definition}

Entropy pairs can be used to characterize CPE and UPE. 
An equivalent statement of the following results was proved in \cite{blanchard93} (also see \cite[Theorem 12.30]{kerr2016ergodic}).   
\begin{theorem}
\label{thm:cpecharac}
A TDS has CPE if and only if $\Gamma^{\alpha}(E(X,T))=X^{2}$ where
$\alpha$ is the entropy rank of $(X,T)$.
\end{theorem}

\begin{proposition}

\label{prop:characUPE}
A TDS has UPE if and only if $E(X,T)=X^2$. 
\end{proposition}

That is, a TDS has UPE if and only if it has CPE and entropy rank 0. In fact, if $(X,T)$ has CPE then $\Delta_X\subseteq E(X,T)$ so in this case to prove entropy rank 0, we only need to show that $E(X,T)=E(X,T)^+$.

Given a compact metrizable space $X$, we define $\text{TDS(X)}$ as the set of all continuous functions from $X$ into $X$ endowed with the  {uniform topology generated by the sup metric}. 

 Note that $\text{TDS(X)}$ is a Polish space. We also define the following subspaces
\[
\textup{UPE}(X)=\{T\in \textup{TDS(X)}: (X,T)\textup{ has UPE}\}
\]
\[
\textup{CPE}(X)=\{T\in \textup{TDS(X)}: (X,T)\textup{ has CPE}\},\textup{ and}
\]
\[
\textup{Mix}(X)=\{T\in \textup{TDS(X)}: (X,T)\textup{ is mixing}\}.
\]

\subsection{Descriptive set theory}
Descriptive set theory and dynamical systems have a longstanding relationship. A few years after Halmos proved that the collection of weak-mixing measure-preserving systems is a dense $G_{\delta}$ set \cite{halmos1944general,halmos2017lectures}, Rohlin proved that the collection of mixing measure-preserving systems is meager. While not constructive, this provided the first proof that there exist weak-mixing but non-mixing measure-preserving systems. Descriptive set theory can also be used to prove \emph{anti-classification} results. This is done by proving that a certain property or equivalence relation cannot be characterized using countable resources. 

In this section we recall the basics of classical descriptive set theory. We refer the reader to \cite{Kechris} for more details. 

Let $\mathcal{X}$ be a \textbf{Polish} space (separable completely metrizable topological space). 
A subset of $\mathcal{X}$ is \textbf{Borel} if it belongs to the smallest $\sigma$-algebra generated by the open sets. Borel sets naturally fall into the Borel hierarchy:
\[ \left .
\begin{matrix} 
    \Sigma^0_1 & \Sigma^0_2  \ldots & \Sigma^0_{\alpha}  \ldots \\
    \Pi^0_1 & \Pi^0_2  \ldots & \Pi^0_{\alpha}  \ldots 
\end{matrix}
\ \ \ \ \ \right \} \alpha < \omega _1,
\]
where $\Sigma^0_1$ is the collection of all open sets, $\Pi^0_{\alpha}$ is the collection of sets whose complement is in $\Sigma^0_{\alpha}$ and $\Sigma^0_{\alpha}$ is the collection of  countable unions of sets  from $\Pi^0_{\beta}$, $\beta < \alpha$. Moreover, $\Sigma^0_{\beta}\cup \Pi^0_{\beta} \subseteq \Sigma^0_{\alpha}\cap \Pi^0_{\alpha}$ whenever $\beta < \alpha$, i.e., the hierarchy is increasing. If $\mathcal{X}$ is uncountable, then the hierarchy is strictly increasing, i.e., there are sets which belong to the $\alpha$ levels but not the previous ones.

Beyond the Borel hierarchy sits the projective hierarchy. We are particularly interested in the first level of the projective hierarchy, i.e., analytic and coanalytic sets. 
We say that a subset of a Polish space is \textbf{analytic} if it is the continuous image of a Borel subset of a Polish space and \textbf{coanalytic} (or, equivalently, $\Pi_1^1$) if it is the complement of an analytic set.
All Borel subsets of a Polish space are both analytic and coanalytic. Moreover, if a set is both analytic and coanalytic, then it must be Borel. However, in every uncountable Polish space there are analytic, and hence coanalytic, sets which are not Borel. Loosely speaking, if a set is analytic or coanalytic but not Borel, it can not be described with countably many quantifiers over a countable set.

A standard method to prove a coanalytic set is not Borel is to reduce it to a known combinatorial set which is not Borel. More specifically, if $\mathcal{B}$ is a known non-Borel subset of some Polish $\mathcal{Y}$, $\mathcal{A} \subseteq  \mathcal{X}$ and $f:\mathcal{Y}\rightarrow \mathcal{X}$ is a Borel function such that $f^{-1}( \mathcal{A}) = \mathcal{B}$, then $\mathcal{A}$ is not Borel. In this case, we say that $\mathcal{B}$ is \textbf{Borel reducible} to $\mathcal{A}$. This inspires the following definition. 
\begin{definition}
A coanalytic set $\mathcal{A}$ subset of a Polish space $\mathcal{X}$ is \textbf{complete coanalytic} (or $\Pi_1^1$-complete) if for every coanalytic set $\mathcal{B}$ of a Polish space $\mathcal{Y}$ there exists a Borel function $f:\mathcal{Y}\rightarrow \mathcal{X}$ such that $f^{-1}( \mathcal{A}) = \mathcal{B}$. 
\end{definition}
In some sense complete coanalytic set are as complicated as coanalytic sets can be. The following proposition simply follows from the definition and the fact that the composition of Borel functions is Borel. 

\begin{proposition}
\label{prop:coanalyticbasic} Let $\mathcal{A}$ be a coanalytic subset of a Polish space $\mathcal{X}$ and $\mathcal{B}$ be a complete coanalytic subset of a Polish space $\mathcal{Y}$. 
If there exists a Borel function $f:\mathcal{Y}\rightarrow \mathcal{X}$ such that  $f^{-1}( \mathcal{A}) = \mathcal{B}$, i.e., $\mathcal{B}$ is Borel reducible to $\mathcal{A}$, then $\mathcal{A}$ is also complete coanalytic. 
\end{proposition}
A classical combinatorial example of complete coanalytic set is, WF, the set of all well-founded trees on $\mathbb{N}$, \cite[p.243 ]{Kechris}. In practice, one normally Borel reduces WF to a given coanalytic set to show that the given coanalytic set is complete coanalytic. In our case, we will use a result of Hurewicz to obtain a combinatorial complete coanalytic set, i.e., the set of all countable compact subsets of a Cantor space is complete coanalytic \cite[p. 245]{Kechris}. In particular, in Theorem~\ref{intcomplete}, we will reduce this set of Hurewicz to prove that $\textup{CPE}([0,1]^d)$ is $\Pi^1_1$-complete. 

An alternate method for proving that a coanalytic set is not Borel is to use the so called ``rank method". Often coanalytic sets in natural settings admit a natural rank, i.e., an assignment which associates a countable ordinal to each element. If this assignment is a $\Pi^1_1$-rank and the rank is unbounded, then the set under question is not Borel. More precisely, we have the following definition. 

\begin{definition}\label{pi11}
Let ${\mathcal X}$ be a Polish space, ${\mathcal C} \subseteq {\mathcal X}$ and $\varphi : {\mathcal C} \rightarrow \omega_1$. We say that $\varphi$ is a  {\bf  $\Pi^1_1$-rank}  if ${\mathcal C}$ is $\Pi^1_1$ and  there are relations ${\mathcal P}, {\mathcal Q} \subseteq {\mathcal X}^2$, one of them $\Sigma^1_1$ and the other $\Pi^1_1$, such that for all $y \in {\mathcal C}$ we have that 
\begin{gather*}
    \{x \in {\mathcal C}: \varphi (x) \le \varphi(y)\} = \{x\in {\mathcal X}: (x,y) \in {\mathcal P} \} = \{x \in {\mathcal X}: (x,y) \in {\mathcal Q}\}.
\end{gather*}
Loosely speaking, $\varphi$ is a $\Pi^1_1$-rank if $\{x: \varphi (x) \le \varphi(y)\} $ is "uniformly Borel in $y$".
\end{definition}
\begin{theorem}\label{overspill}\cite[Section 35.E]{Kechris}
Let ${\mathcal C}$ be a $\Pi^1_1$ set and $\varphi$ be a $\Pi^1_1$-rank on ${\mathcal C} $. If ${\mathcal A}  \subseteq {\mathcal C}$ is $\Sigma^1_1$, then $\varphi$ is bounded on ${\mathcal A} $, i.e., there exists $\alpha < \omega_1$ such that $\varphi(x) < \alpha$ for all $x \in {\mathcal A} $. In particular,

\[ {\mathcal C}  \mbox{ is Borel} \iff \varphi \mbox{ is bounded on } C.
\]
\end{theorem}

{A classical $\Pi^1_1$-rank  known as the Cantor-Bendixson rank will be useful for
us. Given a compact set $A$, we let $A'$ denote the \textbf{Cantor-Bendixson derivative} of $A$, i.e., the set of all limit-points of $A$.} For countable ordinal $\alpha$, $A^{\alpha}$, the $\alpha^{th}$ Cantor-Bendixson derivative of $A$, is defined by transfinite recursion as follows
\[
A^{\alpha +1}=(A^{\alpha})'\text{ and}
\]
\[
A^{\beta}=\cap_{\alpha<\beta} X^{\beta} \text{ for a limit ordinal $\beta$.}
\]


Of course, from separability we have that for each compact set $A$ there is a countable ordinal $\alpha$ such that $A^{\alpha} = A^{\alpha+1}$. The least such $\alpha$ is called the \textbf{Cantor-Bendixson rank of $A$} and is denoted by $|A|_{CB}$. Moreover, we define $A^{\infty} = A^{|A|_{CB}}$.  For a fixed compact metrizable space $X$, the set of all countable compact subsets of $X$ is a $\Pi^1_1$ set and the Cantor-Bendixson rank is a $\Pi^1_1$-rank on this set.  \cite[Section 6]{Kechris} 

Returning to our particular case of $\textup{CPE}(X)$, the following results will allow us to use the rank method. 
The first result is probably known but we include it for the sake of completeness.
\begin{proposition}\label{cpeiscoanalytic}
Let $X$ be a compact metrizable space. Then, $\textup{CPE}(X)$  is coanalytic. 
\end{proposition}
\begin{proof}
 Vetokhin showed \cite{Vetokhin} that the assignment which takes a TDS to its entropy is Borel (in fact, of Baire class two). This implies that \[
 \{T\in \textup{TDS(X)} : (X,T)\textup{ s.t. } h_{top}(X,T)>0\}
 \]
 is Borel. By the definition, we have that a TDS has CPE if and only if all its non-trivial factors have positive entropy. As CPE is characterized by the quantifier $\forall$ on a Borel condition, we obtain that  $\textup{CPE}(X)$ is coanalytic. 
\end{proof}

The following result was proved in the context of effective descriptive set theory \cite[Corollary 2]{westrick2019topological}. Using a notion of Borel expansion we give an alternate classical descriptive set theoretic proof in \cite{usarxiv}. 
\begin{theorem}\label{rankpi11}
Let $X$ be a compact metrizable space. The entropy rank is a $\Pi^1_1$-rank on $\textup{CPE}(X)$.
\end{theorem}

Examples of TDSs of arbitrarily high entropy rank have been found for Cantor maps \cite{barbieri2020} and subshifts \cite{salo2019entropy}. These examples are not transitive (let alone mixing). 
By applying the rank method to the coanalytic set, $\textup{Mix}(X) \cap \textup{CPE}(X)$ ($X$ being a Cantor space), we will see in Theorem~\ref{mainmixincpe} that it is not Borel.

\section{When CPE and UPE are equivalent}
In this section we show that UPE is always a Borel property and we will explain how in some situations CPE and UPE are equivalent properties. As mentioned before, any TDS with UPE also has CPE. We show that any TDS with the the shadowing property that has CPE also has UPE. It is known that for graph maps, the properties of being mixing and having UPE are equivalent \cite[Page 107]{ruette2017chaos}. Hence, for mixing graph maps, we also have that having CPE and having UPE coincide. 

The fact that the class of TDSs that have UPE is Borel can be proved using rank theory, i.e., the collection of objects whose rank is bounded under a $\Pi^1_1$-rank forms a Borel set. We give an alternative proof which may be useful in determining the exact descriptive complexity of the class of maps that have UPE on a given space. While in this paper we won't study the complexity of Borel sets, we note that Borel complexity has been studied in dynamical contexts in \cite{clemens2009isomorphism,gao2016group,jackson2020borel}.

We denote with $K(X)$, the collection of nonempty compact subsets of $X$ endowed with the Hausdorff metric. For a definition of the Hausdorff metric see \cite[Section 4.5.23]{engleking1989general}. Note that if $X$ is compact then so is $K(X)$.
Let $\mathcal{U}=\{U_{1},...,U_{n}\}$ be a collection of open sets of $X$. We
define \[{A}_{\mathcal{U}}\mathcal{=}\left\{  B\in K(X):B\subseteq 
\cup_{i=1}^{n}U_{i}\textup{ and }B\cap U_{i}\neq\emptyset\textup{ for all }1\leq
i\leq n\right\} .\] We have that $\{A_{\mathcal{U}}\}_{\mathcal{U}}$
forms a basis for the topology on $K(X)$ induced by the Hausdorff metric. 

\begin{proposition}\label{Borelext}
Let $X$ be a compact metrizable space and let $$E: \textup{TDS(X)} \rightarrow K(X \times X)$$ be given by $E(T) = E(X,T)$. Then, $E$ is a Borel map.
\end{proposition}
\begin{proof}
Let $U,V$ be open in $X$. We will first observe that
\[\{ T \in \textup{TDS(X)}:  \ E(T) \cap (U \times V)\neq \emptyset\}\tag{\(\dagger\)}
\]
is Borel. Indeed, using an equivalent definition of independence given in \cite[Lemma 3.2]{KerrLiMA} we have that $\dagger$ is satisfied by $T$ if and only if there is a rational number $ r >0 $ such that for all $l \in \N$ there is an interval $I \subseteq \N$ with $|I| \ge l$ and a finite set $F \subseteq I$ with $|F| \ge r |I|$ such that $F$ is an independent set for $(U,V)$.
It is easy to verify that for fixed $U,V, r, l, I, F$ set
\[
   \{T \in \textup{TDS(X)}: F \textup{ is an independent set for } (U,V) \textup{ for } T  \tag{\(\ddagger\)} \}  
\]
is open. Now the set in $\dagger$ is the result of a sequence of  countable union and countable intersections of sets of type $\ddagger$. Hence $\dagger$ is Borel. Since $X$ has a countable basis, by taking unions, we have that $\dagger$ is Borel when $U \times V$ is replaced by any open set $W \subseteq X \times X$. Every closed set in $X \times X$ is the monotonic intersection of a sequence of open sets in $X \times X$. This and the fact that $E(T)$ is closed imply that $\dagger $ is Borel when $U \times V$ is replaced by a closed set $C \subseteq X \times X$. Reformulating the last statement, we have that for all open $W \in X \times X$, the set 
\[\{ T \in \textup{TDS(X)}:  \ E(T) \subseteq W \}\tag{\(\diamond\)}
\]
is Borel. Putting $\dagger$ and $\diamond$ together, we have that 
\[\{T \in \textup{TDS(X)}:  E(T) \subseteq \cup_{i=1}^n  W_i \ \  \&  \  E(T) \cap W_i \neq \emptyset, 1 \le i \le n \}
\]
is Borel whenever $W_1, \ldots,W_n$ are open in $X^2$, completing proof.

\end{proof}

Now using Proposition \ref{prop:characUPE}, we obtain the following. 
\begin{corollary}
\label{cor:upeborel}
Let $X$ be a compact metrizable space. We have that $\textup{UPE}(X)$ is Borel. 
\end{corollary}

Let $(X,T)$ be a TDS and $\delta >0$. We say $\left\{ x_{n}\right\} _{n\in 
\mathbb{N}}\subseteq X$ is a $\delta $\textbf{-pseudo orbit} if $d(T(x_{n}),x_{n+1})
\leq \delta $ for every $n\in \mathbb{N}$ (if $x_n$ is indexed on a finite interval instead of $\mathbb{N}$ we say it is a \textbf{finite $\delta$-pseudo-orbit}). \ We say $(X,T)$ has the \textbf{(finite) 
shadowing property}, if for every $\varepsilon
>0$, there exists $\delta >0$ such that for every (finite) $\delta $-pseudo orbit$,
\left\{ x_{n}\right\} _{n\in \mathbb{N}}$, there exists $y\in X$ such that $
d(x_{n},T^{n}(y))\leq \varepsilon $. The shadowing property is also known as the \textit{pseudo-orbit tracing property}. In the context of this paper (compact metrizable spaces), the finite shadowing property is equivalent to the shadowing property.

\begin{proposition}\label{shadowing}
Let $(X,T)$ be a TDS with the shadowing property. Then, $E(X,T)=E(X,T)^+$. 
\end{proposition}

\begin{proof} Let $(a_{1},a_{2}),(a_{2},a_{3})\in E(X,T)$ and $A_{1}$, $A_{3}$ open sets, such that $a_{1}\in A_{1}$ and $a_{3}\in A_{3}$. Let $\varepsilon >0$,
such that $B(a_{1},2\varepsilon )\subseteq A_{1}$ and $B(a_{3},2\varepsilon
)\subseteq A_{3}$, and $\delta >0$ be a witness for the shadowing property of $(X,T)$
with respect to $\varepsilon $. We may assume, without loss of generality, that $\delta <\varepsilon $. Since $
(a_{1},a_{2}),(a_{2},a_{3})\in E(X,T)$, there exist $n_{1},n_{3}>1$ and 
\[
\left\{ y{(i,j)}\right\} _{(i,j)\in 
\left\{ 1,2,3\right\} ^{2}} \subseteq X,
\]
such that \[y{(i,j)}\in T^{-n_{1}}B(a_{j},\delta
/2)\cap B(a_{i},\delta /2)\] for every $(i,j)\in \left\{ 1,2\right\} ^{2}\setminus \{(2,2)\}$,
\[ y{(i,j)}\in T^{-n_{3}}B(a_{j},\delta /2)\cap B(a_{i},\delta /2)\] for
every $(i,j)\in \left\{ 2,3\right\} ^{2}\setminus \{(2,2)\}$ and 
\[ y_i{(2,2)}\in T^{-n_{i}}B(a_{2},\delta /2)\cap B(a_{2},\delta /2)\] for every $i\in \{1,3\}$.

Let $N=2n_1 n_3$. We will show that $\{Nk\}_{k\in \N}$ is an independence set for $(A_1,A_3)$, yielding that $(a_1, a_3) \in E(X,T)$. For this we will show that given $f:\left\{ Nk\right\} _{k\in \mathbb{N}}\rightarrow \left\{ 1,3\right\}$, we can find a $y \in X$, depending on $f$, such that $T^{Nk}(y) \in A_{f(Nk)}$ for all $k \in \N$, i.e., $y \in \cap_{k \in \N} T^{-Nk}(A_{f(Nk)}) \neq \emptyset$. The idea of the proof is to find a (full) $\delta $-pseudo-orbit which approximates $\{a_{f(kN)}\}_{k \in \N}$  on $\{Nk\}_{k\in \N}$. Then, the $y\in X$ which shadows this $\delta $-pseudo-orbit is the desired point.

First we claim that there exist four finite $
\delta $-pseudo-orbits on $[0,N]$, $\{x_{m}{(i,j)}\}_{m \in [0,N]}$ with $i,j\in\{1,3\}$, such that $
x_{0}{(i,j)}\in B(a_{i},\delta /2)$ and $x_{N}{(i,j)}\in B(a_{j},\delta /2)
$. Note that if we \textit{paste} two of these finite pseudo-orbits $x_m(i,j)$ and $x_{m}(j,i')$ we obtain a finite $\delta$-pseudo-orbit of size $2N$. 
Let $m\in \mathbb{N}$. There exists $k\geq 0$ such that $m\in [Nk,N(k+1))$.
We define \[x_{m}=x_{m \bmod {N}}{(f(Nk),f(Nk+N))}.\] With the pasting finite pseudo-orbits argument, one can check that $\left\{ x_{m}\right\} _{m\in \mathbb{N}}$ is
a $\delta$-pseudo-orbit. Using the shadowing property, there exists $y\in X$ such that 
\[
d(x_{n},T^{n}(y))\leq \varepsilon 
\]
for all $n\in \mathbb{N}$. Thus, for every $m\in \left\{ Nk\right\} _{k\in \mathbb{N}}$ we
have that \[T^{m}(y)\in B(a_{f(m)},\varepsilon +\delta )\subseteq A_{f(m)},\] yielding that $(a_{1},a_{3})\in E(X,T)$. 

It remains to show that the claimed finite pseudo-orbits exist. 

For $i\in \left\{ 1,3\right\}$, we define $
x_{m}{(i,i)}=T^{m\bmod{n_i}}(y{(i,i)})$, $m \in [0,N]$. From the fact that 
\[ y{(i,i)}\in T^{-n_{i}}B(a_{i},\delta /2)\cap B(a_{i},\delta /2) \iff T^{n_i} (y{(i,i)}) \in B(a_{i},\delta /2)\cap T^{-n_{i}} (B(a_{i},\delta /2)),\] 
it follows that the distance between $y(i,i)$ and $T^{n_i}(y (i,i))$ is less than $\delta$ and  $\{x_{m}{(i,i)}\}_{m \in [0,N]}$ is finite $\delta $-pseudo-orbit.

For $(i,j)\in  \left\{ (1,3), (3,1) \right\}$ and $m \in [0,N]$, we define
\[x_{m}{(i,j)}=
\begin{cases}
T^{m}(y{(i,2)}) & m\in \lbrack 0,n_{i}) \\ 
T^{m\bmod {n_{i}}}(y_i{(2,2)}) & m\in \lbrack n_{i},n_{i}n_{j}) \\ 
T^{m\bmod {n_{j}}}(y{(2,j))} & m\in \lbrack n_{i}n_{j},n_{j}n_{i}+n_j) \\ 
T^{m\bmod {n_{j}}}(y{(j,j)}) & m\in \lbrack n_{j}n_{i}+n_j,2n_{i}n_{j}].

\end{cases}
\]
Considering how $y(i,j)$ were constructed, we have that $x_{0}{(i,j)}=y(i,2)\in B(a_{i},\delta /2)$, and $x_{N}{(i,j)}=x_{2n_in_j}{(i,j)}=y(j,j)\in B(a_{j},\delta /2)$. Furthermore, one can check that it is indeed a finite $\delta$-pseudo orbit. There are only four places where we have to check that the $\delta$-pseudo orbit condition holds. First, note that since $Tx_{n_i-1}=T^{n_i}(y(i,2))\in B(a_{2},\delta /2)$ and $x_{n_i}=y_i(2,2)\in B(a_{2},\delta /2)$, then $d(Tx_{n_i-1},x_{n_i})\leq \delta$. The proofs for the jumps at $n_in_j$, $n_j(n_i+1)$, and $2n_in_j$ are similar.  
\end{proof}

\begin{corollary}\label{upeshadow}
Let $(X,T)$ be a TDS with the shadowing property. Then $(X,T)$ has UPE if and only if it has CPE. 
\end{corollary}
\begin{proof}
If a TDS has UPE then necessarily it has CPE. 
Assume that $(X,T)$ has CPE, by Theorem \ref{thm:cpecharac} there is a countable ordinal $\alpha$ such that 
\[
\Gamma^{\alpha}(E(X,T)) = X^2.
\]
By Proposition~\ref{shadowing}, it follows that $\alpha=0$, so by Proposition \ref{prop:characUPE}, $(X,T)$ has UPE.
\end{proof}

We refer the reader to  \cite[Theorem 3.8]{kulczycki2014} for other characterizations of chaotic properties of TDSs with the shadowing property. 

Given a compact metrizable space $X$ we define 
\[
\textup{Shad}(X)=\{T\in \textup{TDS}(X):(X,T)\textup{ has the shadowing property}\}.
\]

\begin{proposition}\label{shadowborel}
Let $X$ be a compact metrizable space. Then $\textup{Shad}(X)$ is Borel. 
\end{proposition}
\begin{proof}
Recall that for compact metrizable spaces, the finite shadowing property  is equivalent to the shadowing property. Moreover, given a compact metrizable space $X$ and a dense subset $D \subseteq X$, we have that $(X,T)$ has the finite shadowing property if and only if $(X,T)$ has the finite shadowing property with pseudo-orbit as well as the tracing point taken in $D$. That is, $T \in \textup{Shad}(X)$ if and only if
\begin{align*}
    & \forall n\in\N  \  \exists m  \in \N\textup { such that} \\   & \left [ \forall p \in \N, \  \{y_i\}_{i=0}^p \in D  \textup { with }   d(T(y_i), y_{i+1}) < 1/n , \  \forall 0 \le i \le p-1   \right ]  \\
    & \exists y \in D \textup{ such that } d(T^i(y), y_i) < 1/m, \  \forall 0 \le i \le p.
\end{align*}
(See the proof of \cite[Lemma 3.1]{fernandez2016shadowing} for an analogous characterization.)
In the above reformulation, all of the quantifiers are over countable sets and the condition $d(x,y) < \epsilon$ is open, implying that $\textup{Shad}(X)$ is Borel.  
\end{proof}

\begin{corollary}\label{shadowcpeborel}
Let $X$ be a compact metrizable space. Then $\textup{Shad}(X)\cap \textup{CPE}(X)$ is Borel. 
\end{corollary}
\begin{proof}
From Proposition~\ref{upeshadow}, we have that  $\textup{Shad}(X)\cap \textup{CPE}(X) = \textup{Shad}(X)\cap \textup{UPE}(X)$. Using Propositions~\ref{shadowborel}  and Corollary~\ref{cor:upeborel}, we conclude our result.
\end{proof}

A TDS, $(X,T)$, is \textbf{algebraic} if $X$ is a compact abelian group and $T$ a group automorphism. An algebraic TDS has UPE if and only if it has CPE \cite{chungLi2015}. For recent results regarding the local entropy theory of algebraic actions see \cite{barbieri2019markovian}. 

On mixing graph maps we also see the same phenomenon. We say that a set $X$ is a \textbf{topological graph} if there exist a  finite number of pairwise disjoint intervals $[a_n,b_n]$ with $0<n\leq N$, and $Q \subseteq \{a_1,...,a_N,b_1,...,b_N\}^2$ such that $X=\cup_n[a_n,b_n]/Q$. 
We say that a TDS is a \textbf{graph map} if $X$ is a topological graph. 
\begin{theorem}
\label{thm:blokh}
\cite{blokh1984transitive},\cite[Page 107]{ruette2017chaos} Let $(X,T)$ be a graph map. Then $(X,T)$ is mixing if and only if it has UPE.
\end{theorem}

\begin{remark}
\label{rem:mix}
To prove that a TDS is mixing it is enough to check the conditions on a countable basis of the topology. With this, one can see that for any compact metrizable space $X$, $\textup{Mix}(X)$ is Borel. 
\end{remark}
 
\begin{corollary}\label{cor:graphBorel}
Let $X$ be a topological graph. Then $\textup{Mix}(X)\cap \textup{CPE}(X)$ is Borel.
\end{corollary}
\begin{proof}
This follows from the fact that $\textup{Mix}(X)$ is Borel, Theorem~\ref{thm:blokh} and Corollary~\ref{cor:upeborel}
\end{proof}

\section{Complexity of systems with CPE on manifolds and graphs}
In this section we show that the collection of TDSs on $X=[0,1]^d$ having CPE is complete coanalytic. We also show how one can adapt the proof for other compact orientable topological manifolds and topological graphs.

Throughout this section, $I =[0,1]$.
\begin{theorem}\label{intcomplete}
We have that $\textup{CPE}(I^d)$ is complete coanalytic.
\end{theorem}
\begin{proof}
We will first show the proof for $d=1$. Once complete, it will be easily lifted to the case of $X= I^d$.

We recall a classical theorem of Hurewicz:
$\cQ$, the collection  of countable compact subsets of $I$, is a  $\Pi^1_1$-complete subset of $K(I)$, \cite[Theorem 27.5]{Kechris}. In light of  Proposition~\ref{prop:coanalyticbasic}, it suffices to construct a continuous function $\psi: K(I) \rightarrow \text{TDS}(I)$ such that $\psi(A)$ has CPE if and only if $A$ is countable. 

In this proof we fix $T:I \rightarrow I$, as
\[
T(x) =
\begin{cases}
3x & 0\leq x\leq1/3\\
2-3x & 1/3<x\leq2/3\\
3x-2 & 2/3<x\leq1.
\end{cases}
\]
It is easy to check that $(I,T)$ is mixing and hence, using Theorem \ref{thm:blokh}, $E(I,T)=I^2$. For a closed interval $J =[a,b]$, let $T_J: J \rightarrow J$ be the scaled copy of $T$ on interval $J$. Again, we have that $E(J,T_J)= J \times J$.

For each $A \in K(I)$, we let $\C(A)$ be the collection of all intervals contiguous to $A$, i.e, the collection of maximal connected components of $[0,1]\setminus A$.

Now we construct our function $\psi: K(I) \rightarrow \text{TDS}(I)$. Let $A \in K(I)$, and $\psi(A) : [0,1] \rightarrow [0,1]$ be a function that is the identity on $A$ and $T_{\overline{J}}$ on $\overline{J}$ for each $J \in \C(A)$. It is easily seen that $\psi(A)$ is a well-defined  element of $\text{TDS}(I)$. Moreover, $\psi$ is continuous.

For every $J \in \C (A)$, we have that $\psi(A)(\overline{J})=\overline{J}$. This implies that if $x\in \overline{J}$ and $y\notin \overline{J}$ then $(x,y)\notin E(\psi (A))$. Furthermore, as the restriction of  $\psi(A)$ to $\overline{J}$ is a scaled copy of $T$, we have that $\overline{J} \times \overline{J} \subseteq E(I,\psi(A))$, implying that
\[ E(I,\psi(A)) =  \bigcup _{J \in \C(A)} \left (\overline{J} \times \overline{J} \right ) \cup \Delta_I.
\]
Now, using transfinite induction and basic properties of Cantor-Bendixson derivatives, we have that for each countable ordinal $\alpha$, \[ \Gamma^{\alpha} (E(I,\psi(A))) =  \bigcup _{J \in \C(A^{\alpha})} \left (\overline{J} \times \overline{J} \right ) \cup \Delta_I.
\]  Where $A^{\alpha}$ is the $\alpha^{th}$ Cantor-Bendixson derivative of $A$. Hence,
\[
\Gamma^{\infty}(E(I, \psi(A)))=\bigcup _{J \in \C(A^{\infty})} \left (\overline{J} \times \overline{J} \right ) \cup \Delta_I.
\]
Using Theorem \ref{thm:cpecharac} and the fact that $A$ is countable if and only if $A^{\infty}$ is empty,  we conclude that $(I,\psi(A))$ has CPE if and only if $A$ is countable, completing the proof of the case of $d=1$.

For the case of arbitrary $d$, we define $\psi_d : K(I) \rightarrow \text{TDS}(I^d)$ by simply taking the countable product, i.e.,
\[\psi_d(A) = \psi(A)^d.
\]
By Theorem \ref{thm:basic} (4), we have that $(I^d, \psi(A)^d)$ has CPE if and only if $(I, \psi(A))$ has CPE; this implies that $(I^d,\psi_d(A))$ has CPE if and only if $A$ is countable. This completes the proof.
\end{proof}

\begin{remark} \label{rem:topmanifolds}
 Observe that for each $A \in K(I)$, we have that $\psi_d(A)$ is a functions which takes the boundary of $I^d$ onto the boundary of $I^d$. Identifying the boundary of $I^d$ to a single point, we obtain that $CPE(S^d)$ is complete coanalytic where $S^d$ is the $d$-dimensional sphere. By taking finite products of $S^1$, we obtain that $CPE(\mathbb{T}^d)$ is complete coanalytic where $\mathbb{T}^d$ is the $d$-dimensional torus. Using these techniques, one can show that $CPE(X)$ is complete coanalytic  where $X$ is a topological graph or a compact orientable manifold.
\end{remark}

 The systems constructed in this section are not even transitive. Obtaining mixing systems will require an intricate, lengthier construction which will be carried out in the next section.
 

\section{Complexity of mixing systems with CPE}\label{sec:mixingNonBorel}
In this section, we will prove the main result of the paper, that is, we will show that the set of all mixing TDSs having CPE on a Cantor space is a coanalytic set that is not Borel. First we give a brief summary of the outline of the proof.  

Let $C \subseteq [0,1]$ be the Cantor set, and $\Kcal$ be the collection of compact subsets of $(0,1) \cap C$. For each $A \in \Kcal$, we will construct a mixing TDS $(X_A,T_A)$, so that $(X_A,T_A)$ has CPE if and only if $A$ is countable.  Moreover, $X_A$ will be homeomorphic to the Cantor set for all $A \in \K $. Although, this assignment, $\Kcal\rightarrow \textup{TDS}(C)$, is rather natural, it does not seems  be Borel. Hence, we will rely on the rank method to prove that the collection of mixing maps having CPE is not Borel, i.e., we will show that for each countable ordinal $\alpha$ there is system $(X_A,T_A)$ having CPE so that the entropy rank of $(X_A,T_A)$ is greater than $\alpha$. (If the assignment was Borel we would get complete coanalytic.) 

The construction of $(X_A,T_A)$, $A \in \K$, is rather intricate and involved. It will be carried out in two parts. In Section~\ref{topcon}, we will construct, for each $A \in \K$, $(Y_A, R_A, B_A)$ so that  $(Y_A, R_A)$ is a topological dynamical system and $B_A \subseteq Y_A^2$ is closed; moreover, it will have the property that $\Gamma^{\infty} (B_A) = Y_A ^2$ if and only if $A$ is countable. In addition, we will show that for each countable ordinal $\alpha$ there is a set $A \in \K$ so that the $\Gamma$-rank of $B_A$ is greater than $\alpha$. 

In Section~\ref{sec:2}, we will transform $(Y_A, R_A, B_A)$ to a mixing TDS having CPE $(X_A,T_A)$, 
so that $B_A$ leads to the entropy set of $T_A$ in a natural way, i.e., $\Gamma ^{\infty}  (B_A) = Y_A^2$ if and only $\Gamma ^{\infty}  ((E(X_A,T_A)) = X_A^2$, yielding that the $\Gamma$-rank of $B_A$ is the same as the entropy rank of $(X_A,T_A)$ whenever $A \in \K$ is countable. We also show in this section that $X_A$ is a Cantor space. Finally, putting it all together we will arrive at our main result.

We need two fundamental dynamical systems, $(X,T)$ and a subsystem $(Y,R)$ ($R$ is simply the restriction of $T$ to $Y$), to begin our construction.
Below we establish their existence. 

\begin{lemma}\label{startup}
There exists a  mixing TDS $(X,T)$ and  a mixing subsystem $(Y,R)$ such that 
\begin{itemize}
\item $X$ and $Y$ are Cantor spaces.
    \item $(X,T)$ is UPE, i.e. $E(X,T) = X^2$,
    \item $E(Y,R)=\Delta_{Y}$, and
    \item $C$ is a subset of $fix(Y,R)$, the set of fixed points of $(Y,R)$. 
\end{itemize}
\end{lemma}
\begin{proof}
Let $(X_{0},T_{0})$ be the full shift on $\{0,1\}$. Hence, it is mixing and $E(X_{0},T_{0})=X_{0}^{2}$. By \cite{tworesults,weiss1971topological} there is a mixing subshift $(Y_0,R_0)$ of $(X_{0},T_{0})$ with entropy zero, two fixed points and a dense set of periodic points. As $(Y_0,R_0)$ has zero entropy,  by Theorem \ref{thm:basic} (1), we have that  $E(Y_{0},R_{0}) \subseteq \Delta_{Y_{0}}$. Furthermore, using the density of periodic points, one can conclude that
$E(Y_{0},R_{0})=\Delta_{Y_{0}}$. Let $(X,T)$ and $(Y,R)$ be the  countable products of $(X_0,T_0)$ and $(Y_{0},R_{0})$, respectively. As the product of mixing systems is mixing, we have that $(X,T)$ and $(Y,R)$ are mixing. Moreover, $X$ and $Y$ are homeomorphic to the Cantor set. Using Theorem \ref{thm:basic} (2) we have that $(X,T)$ has UPE and $E(Y,R)=\Delta _Y$.  Finally, since $(Y_0,R_0)$ has two fixed points, $(Y,R)$ has a perfect set consisting of fixed points. We identify this set with the standard Cantor set $C$.
\end{proof}

For the rest of this section, we work with $(X,T)$ and $(Y,R)$ of Lemma~\ref{startup}.

\subsection{Topological Construction}\label{topcon}

{We now give a sequence of auxiliary results which yields the main result of this section, Proposition \ref{highrank}. This is the most technical subsection of the
paper.}

Let $(Y,R)$ be the TDS of Lemma~\ref{startup}.
If $x, y $ are in the Cantor set $C$, then $x < y$ means that $x< y$ in the ordering on $C$ induced by being a subset of $[0,1]$. 

We first introduce terminology and auxiliary functions. 

Let $A$ be a compact subset $(0,1)$. As before, we say an interval $I\subseteq (0,1)$ is contiguous to $A$ if $I$ is a maximal connected component of $(0,1) \setminus A$.  In particular, a contiguous interval has the form  $(a,b)$ and at least one of $a, b \in A$. Moreover, if neither $a$ nor $b$ belongs to $\{0,1\}$, then $a, b \in A$. We let $\C(A)$ be the set of intervals contiguous to $A$. For an interval $I \subseteq (0,1)$, we denote the left endpoint and right endpoint of $I$ with $\ell(I)$ and $r(I)$ respectively. We let $\overset{\multimapdotinv}{I} = I \cup \{\ell(I)\}$ if $\ell(I) \in A$ and otherwise $\overset{\multimapdotinv}{I} = I$. For the sake of brevity, we use $I[$ and $]I$ to denote the right component and the left component of $(0,1) \setminus I$, respectively.
We let $L(A)$ be the set of points in $A$ which are left endpoints of intervals contiguous to $A$. In particular, $0 \notin L(A)$ and  $\max A \in L(A)$.

We start with various notions essential for our definition of $B_A$.

Let $A \in \K$. For each $\alpha$ less than $|A|_{CB}$, the Cantor-Bendixson rank of $A$, we define a Borel function $m^{\alpha}_A:A \rightarrow A$ in the following way: If $a \in A^{\alpha} \setminus A^{\alpha +1}$, then $m^{\alpha}_A(a)$ is the right endpoint of interval contiguous to $A^{\alpha+1}$ containing $a$. Note that $m^{\alpha}_A(a)$ is the closest point of $A^{\alpha+1}$ to the right of $a$ or $1$. Everywhere else $m^{\alpha}_A$ is the identity function. Finally, let $m_A = \sup_{\alpha < |A|_{CB}} m^{\alpha}_A$. (If $A$ has rank zero, i.e. $A$ is perfect, then we simply take $m_A$ to be the identity.) Note that this function is well-defined and $m_A$ is identity on $A^{\infty}$. 

Let $A \in \K$, $I \in \C(A)$, and $b=r(I)$. 
We define the function $n _{I,A}:  [b,1] \cap A \rightarrow [b,1] \cap A$ by
\[n_{I,A} (c) = \sup_{\alpha : A^{\alpha} \cap [b,c] \neq \emptyset } \inf \left (  A^{\alpha} \cap [b,c] \right ). 
\]
Finally, define $t_{I,A}: [b,1]\cap L(A) \rightarrow A$ by $t_{I,A} (c) = m_A(n_{I,A} (c) )$. Note that indeed $t_{I,A} (c) \in A$, and $b \le t_{I,A} (c)$. For the sake of brevity, when the $A$ is fixed in the statement of the result, we will abuse notation and write $t_{I}$ instead of $t_{I,A}$,  suppressing $A$. Note that when $A\in\K$ is fixed then for every point $a\in L(A)$ there is just one interval $I\in \C(A)$ with $a=\ell (I)$. In this case we define $t_a=t_I$.


\begin{proposition}\label{entuallysame}
Let $A \in \K$, $\alpha$ be a countable ordinal, $I \in \C(A^{\alpha})$, and $J, J' \in \C(A)$ with $J, J' \subseteq I$. Then, $t_J|_{I[} = t_{J'}|_{I[}$.
\end{proposition}
\begin{proof}
Let $c \in L(A) \cap I[$. Then $b=r(I) \in A^{\alpha}$ and $A^{\alpha} \cap [b,c] \neq \emptyset$. Hence, $n_{J,A}(c) = n_{J',A} (c)$, implying that $t_J (c) =m_A(n_{J,A}(c))= m_A(n_{J',A}(c))=t_{J'} (c)$.
\end{proof}
\begin{proposition}\label{perfectbarrier}
Let $A \in \K$ with $A^{\infty} \neq \emptyset$, $I \in \C(A^{\infty})$ and $J \in \C(A)$ with $J \subseteq I$. Then, $t_J(c) = r(I)$ for all $c \in L(A) \cap I[$.
\end{proposition}
\begin{proof}
Let $c \in L(A) \cap I[$. By the definition of $n_{J,A}$, we have that $n_{J,A} (c) = r(I)$. Moreover, $m_A(r(I)) = r(I)$. Hence, $t_J(c) = r(I)$.
\end{proof}
\begin{proposition}\label{fixedhigherlevel}
Let $A \in \K$ with $A^{\alpha} \neq \emptyset$ and  $I \in \C(A^{\alpha +1 })$. There is $f \in Y$ such that for all $J \in \C(A^{\alpha})$ and all $a \in \overset{\multimapdotinv}{J} \cap L(A)$, we have $t_a|_{J[ \cap I} = f$. 
\end{proposition}
\begin{proof}
We will show that $f = r(I)$ satisfies the proposition. To this end, let $a \in \overset{\multimapdotinv}{J} \cap L(A)$ and $c \in L(A) \cap J[ \cap I$. Let $J' \in \C(A)$ such that $a = \ell (J')$. Then, $J'\subseteq J \subseteq I$.  Consider the interval $[r(J'),c]$. $A^{\alpha +1} \cap [r(J'),c] = \emptyset$ and $A^{\alpha} \cap [r(J'),c]  \neq \emptyset$ as $r(J) \in  A^{\alpha} \cap [r(J'),c]$.  Hence, we have that $n_{J',A}(c) \in A^{\alpha} \cap I$.  Now by the definition of $m_A$, we have that $m_A(n_{J',A}(c)) = r(I)$. Hence, $t_a (c) = t_{J',A}(c)= m_A(n_{J',A}(c)) = r(I) =f $, completing the proof.  \end{proof}

For every $A \in \Kcal$ we define $Y_A = \prod_{a \in L(A)} Y$ and the mapping $R_A$ as the product map $\prod_{a \in L(A)} R$.
We are now in a position to define $B_A$. 
To facilitate the notation, for $y, y' \in L(A)$ and any set $I$ we use $y \stackrel  {I} {\asymp} y' $ to denote that $y (a) = y'(a)$ for all $a \notin I$.

Let $a\in L(A)$ with $I \in \C(A)$ such that $a= \ell (I)$. We define
\[B_{I,A} = \{ (y, y')\in  Y_A^2:  y \stackrel  {a} {\asymp} y', \{y(a),y'(a)\} \in R \ \ 
 \& \ \ y|_{I[} = t_a|_{I[} \}.
\]

and 
\[D_{I,A} = \{ (y, y')\in Y_A^2:  y \stackrel  {a} {\asymp} y' \
 \& \ \ y|_{I[} = t_a|_{I[} \}.
\] 
 For the sake of brevity, when we are working with a fixed $A$, we write $B_a$ and $D_a$ instead of $B_{I,A}$ and $D_{I,A}$, respectively, whenever $a = \ell (I)$, $I \in \C(A)$.
 
Finally,  we let 
\[B_A =  \bigcup_{a  \in L(A)}  (B_a  \cup \Delta _{Y_A})   \ \ \ \ \ \& \ \ \ \  \ D_A = \bigcup_{a  \in L(A)} D_{a}.\]
where $\Delta _{Y_A}$ is the diagonal of $Y_A$. Note that $\Delta _{Y_A}\subseteq D_{\max (A)}\subseteq D_A$.

\begin{proposition}\label{closed} Let $A \in \Kcal$. We have that
$B_A$ and $D_A$  are closed.
\end{proposition}
\begin{proof}
 We prove that $B_A$ is closed. The proof for $D_A$ is analogous. First observe that from continuity of $R$ it is easy to verify that $B_a$ is closed for each $a \in L(A)$. Suppose $(y,y')$ is a limit-point of $B_A$. If $y=y'$, then $(y,y')\in \Delta _{Y_A} \subseteq B_A$. If there are $a_1 \neq a_2$ in $L(A)$ with $y(a_i)\neq y'(a_i)$, $i=1,2$, then it is easy to find an open set in $Y_A^2$ containing $(y,y')$ which doesn't intersect $B_A$. Hence, we are in the case that there is exactly one $a \in L(A)$ such that $y(a) \neq y'(a)$. Note that for sufficiently fine neighborhood of $(y,y')$ has the property that it has empty intersection with $D_b$, $b\neq a$. Hence, we have that $(y,y')$ is a limit-point of $B_a$. As $B_a$ is closed, we have that $(y,y') \in B_a \subseteq B_A$, completing the proof. 
\end{proof}
\begin{remark}
\label{ignoreB}
As $(Y,R)$ is transitive, we have that $B_{A}\subseteq D_{A}\subseteq \Gamma(B_{A})\subseteq \Gamma(D_A)$, implying that $\Gamma ^{\infty} (B_{A})=\Gamma ^{\infty}(D_A)$. 
\end{remark}

Now we want to show that $A$ is countable if and only if there is a countable ordinal $\alpha_A$ with $\Gamma^{\alpha_A}(D_A) = Y_A ^2$. \\

We introduce further notation to facilitate our proofs. For $M \subseteq Y_A^2$ and $I$ any set, we let
\[\stackrel {I}{\Box}M =  \{ (w,w') \in Y_A^2: \exists (y, y' )\in M \text{ such that } w \stackrel  {I} {\asymp} y, w' \stackrel  {I} {\asymp}  y' \}.
\]

\begin{remark}
\label{rem:boxbasic}
We note that $ \stackrel {I}{\Box}\stackrel {I}{\Box}M =\stackrel {I}{\Box}M$, $ \stackrel {I}{\Box}\stackrel {J}{\Box}M \subseteq \stackrel {I \cup J}{\Box}M$, and $\stackrel {I}{\Box}M$ is closed provided that $M$ is closed.
Moreover, $\stackrel {I}{\Box} (M_1 \cup M_2)  =\stackrel {I}{\Box}M_1 \cup \stackrel {I}{\Box}M_2$.
\end{remark}

The following lemma is a crucial step for proving that  $\Gamma^{\infty}(A) = Y_A^2$ for countable $A$.

\begin{proposition}\label{Enlarge}
Let $A \in \Kcal$, $I \subseteq (0,1)$  be an interval and $G$ be a countable collection of pairwise disjoint subintervals of $I$ such that $I \cap L(A) \subseteq \cup G$. Furthermore, let $f \in Y$, $t: I[ \cap L(A)  \rightarrow Y$ and $\{H_J\}_{J\in G}$ be a family of closed subsets of $Y_A^2$, such that for every $J\in G$ we have the following properties.

\begin{enumerate}
    \item For all $(y,y') \in H_J$, we have that $y \stackrel  {J} {\asymp}  y'$. 
    \item For all $(y,y') \in H_J$,  $y|_{I \cap J[} =f$, and $y|_{I[} =t$.
     \item For all $w \in Y_A$, there is $(y,y') \in H_J$ such that $w|_{]J} = y|_{]J} $.
    \item $\stackrel {J}{\Box}{H_J}= H_J$.
 
\end{enumerate}
Then,  letting $H= \bigcup_{J \in G}H_J$,
we have that $\stackrel {I}{\Box}H  \subseteq \Gamma (\overline{H})$. 
\end{proposition}
\begin{proof}
 If $I \cap L(A) = \emptyset$, then $\stackrel {I}{\Box}H = H$ and  as $H \subseteq \Gamma (\overline{H})$, we are done. 
 Hence, assume that $I \cap L(A)  \neq \emptyset$. 
 Note that all $(y,y') \in H$ have the property that $y \stackrel  {I} {\asymp}  y'$ and $y|_{I[} =t$. Considering this and Property 3, we obtain that \[\stackrel {I}{\Box}H =\{(x,x'):x\stackrel  {I} {\asymp} x' \text{and }x|_{I[} =t\}.\]   As such,  to conclude that $\stackrel {I}{\Box}H \subseteq  \Gamma (\overline{H})$, it will suffice to show that for all $v_0 \in Y_A$, $F \subseteq I \cap L(A)$ nonempty finite set, and  $u, u': F \rightarrow Y$,  there exists $(w,w') \in H^+$ such that  $w|_{]I} =v_0|_{]I}$ and $w,w'$ are extensions of $u,u'$, respectively. Moreover, as equivalence relations are symmetric and transitive, we may assume that $u'$ is a constant function taking value $f$.  Let $J_1, \ldots, J_n$ be intervals in $G$ so that $F \subseteq \cup_{i=1}^n J_i$ and $J_i$ is to the left of $J_{i+1}$. Fix $i$ for the moment. By applying Properties 1-4 to $J_i$, we have that for all $v \in Y_A$, there exists $(y,y') \in H_{J_i}$ so that 
\begin{itemize}
 \item  $y|_{]J_i} = y'|_{]J_i}= v|_{]J_i}$,
    \item $y'|_{J_i}=f$, $y$ and $u$ agree on $F \cap J_i$, 
    \item $y|_{I \cap J_i[}  = y'|_{I \cap J_i[} =f$  and $y|_{I[}  = y'|_{I[} =t$.
\end{itemize}
Applying the above observation with $i=1$ and $v=v_0$, we  obtain $(w_0,w_1) \in H_{J_1}$ so that $w_0|_{]J_1}=w_1|_{]J_1} =v_0$, $w_0|_{J_1}=f$, $w_1$ is $u$ on $F\cap J_1$, $w_0|_{I \cap J_1[}  = w_1|_{I \cap J_1[} =f$  and $w_0|_{I[}  = w_1|_{I[} =t$.  Now, we apply the observation with $i=2$, and  $v=w_1$.  Hence, we obtain  $(y,w_2) \in H_{J_2}$ so  that 
 $y|_{]J_2}=w_2|_{]J_2} =w_1$, $y|_{J_2}=f$, $w_2$ is $u$ on $F\cap J_2$, $y_0|_{I \cap J_2[}  = w_2|_{I \cap J_2[} =f$  and $y|_{I[}  = w_2|_{I[} =t$.
Note that necessarily $y=w_1$. Hence, we have that $(w_0,w_1), (w_1, w_2) \in H$, i.e., $(w_0,w_2) \in H^+$.  Continuing in this fashion, setting $v=w_{i-1}$ at stage $i$, we arrive at $(w_0,w_n) \in H^+$ so that $w_n =u$ on $F$,  $w_0$ is the constant $f$ on $F$ and $w_0|_{]I}=w_1|_{]I}=v_0$,
completing the proof of $\stackrel {I}{\Box}H \subseteq  \Gamma (\overline{H})$.
\end{proof}

Let $I \subseteq (0,1)$. For the rest of this section we let
\[E_I = \stackrel {I}{\Box} \bigcup _{a \in I \cap L(A)} D_a.
\]
We note that $E_I$ depends on $A$ but $A$ is suppressed in the notation for the sake of  brevity.
\begin{proposition}\label{EI+} 
Let $A \in \Kcal$, $I \subseteq (0,1)$ and $M \subseteq Y_A ^2$ be such that for all $(y,y') \in M$, $y = y'$ on $ L(A) \setminus I$. Then, $ \left (\stackrel {I}{\Box} M \right )^+ = \stackrel {I}{\Box} M$. In particular, 
$E_I^+ = E_I$ for all $I \subseteq (0,1)$.
\end{proposition}
\begin{proof}
To show that $ \left (\stackrel {I}{\Box} M \right )^+ \subseteq \stackrel {I}{\Box} M$,
let $(u,v),(v,w)\in \stackrel {I}{\Box}M$. This implies there exist $(x,x'),(y,y')\in M$ such that $u \stackrel  {I} {\asymp} x$, $v \stackrel  {I} {\asymp}  x'$, $v \stackrel  {I} {\asymp}  y'$ and $w \stackrel  {I} {\asymp}  y'$. Using the hypothesis we conclude that $x' \stackrel  {I} {\asymp}  w$, hence $(u,w)\in \stackrel {I}{\Box}M$. 
\end{proof}

\begin{proposition}\label{closedE}
Let $A \in \Kcal$, $\alpha$ be a countable ordinal and $I = (0,1)$ or $I\in \C(A^{\alpha +1})$. Let $G$ be the set of all  $J \in \C(A^{\alpha})$ such that $J \subseteq I$. Then, $ \bigcup_{J \in G}  E_{\overset{\multimapdotinv}{J}}  \bigcup \Delta_{Y_A} $ is closed. 
\end{proposition}
\begin{proof} We first show that $E_{\overset{\multimapdotinv}{J}}$ is closed for all $J \in \C(A^{\alpha})$. By Proposition~\ref{entuallysame}, there is $t:L(A) \cap J[ \rightarrow A$ such that for all  $(y,y') \in D_a$, $a \in L(A) \cap \overset{\multimapdotinv}{J}$, we have that $y|_{J[} = y'|_{J[} = t$. Now, from the definitions of $D_a$ and $\stackrel{\overset{\multimapdotinv}{J}}{\Box}$ we have that 
\[E_{\overset{\multimapdotinv}{J}} =  \left \{(y,y'): y|_{]\overset{\multimapdotinv}{J} } = y'|_{]\overset{\multimapdotinv}{J} } \ \ \& \ \ y|_{\overset{\multimapdotinv}{J} [} = y'|_{\overset{\multimapdotinv}{J} [}  = t \right \},  \]
implying that $E_{\overset{\multimapdotinv}{J}}$ is closed.

 Let 
$E = \bigcup_{J \in G}  E_{\overset{\multimapdotinv}{J}}\bigcup \Delta_{Y_A} $. To show that $E$ is closed, let $\{(y_n, y'_n)\}$  be a sequence in $\bigcup_{J \in G} E_{\overset{\multimapdotinv}{J}}$  that converges to $(y,y')\in Y_A^2$.  Let $J_n \in G$ such that $ (y_n,y'_n) \in  E_{\overset{\multimapdotinv}{J_n}}$. If there is $J \in G$ such that $J_n = J$ for infinitely many $n$'s, then by the fact that $E_{\overset{\multimapdotinv}{J}}$ is closed we have that $(y,y') \in E_{\overset{\multimapdotinv}{J}} \subseteq E$ and we are done. Hence, we may assume that $J_n$'s are all distinct. Since $ (y_n,y'_n) \in  E_{\overset{\multimapdotinv}{J_n}}$, this implies that for every $p\in L(A)$, $y_n(p)=y'_n(p)$ for all but at most one $n$. Hence, $y=y'$. As $\Delta_{Y_A} \subseteq E$, the proof is complete.
\end{proof}

Let $A \in \K$ and $\alpha$ be a countable ordinal. Then, we define
\[E_{\alpha} = \bigcup_{I \in \C(A^{\alpha})}  E_{\overset{\multimapdotinv}{I}}.\]
\begin{proposition}\label{basicprofEalpha}
Let $A \in \Kcal$. We have that $D_A \subseteq E_{\alpha}$, $E_{\alpha}$ is closed, and if $\alpha$ is a limit ordinal, then $E_{\alpha} = \overline{ \cup_{\beta < \alpha} E_{\beta}}$. 
\end{proposition}
\begin{proof} 
That $D_A \subseteq E_{\alpha}$ follows from the fact that if $J \in \C(A)$, then $J \subseteq I$ for some $I \in \C(A^{\alpha})$. Note that $\Delta_{Y_A} \subseteq D_{\max A} \subseteq E_{\alpha}$. Now that $E_{\alpha}$ is closed follows from Proposition~\ref{closedE}.   

Suppose $\alpha$ is a limit ordinal. For each $\beta < \alpha$, $E_{\beta} \subseteq E_{\alpha}$. As $E_{\alpha}$ is closed, we have the containment $E_{\alpha} \supseteq \overline{ \cup_{\beta < \alpha} E_{\beta}}$. To see the reverse containment, let $(y,y') \in E_{\alpha}$. By the definition of $E_{\alpha}$, there is $ I\in \C(A^{\alpha})$, $a \in \overset{\multimapdotinv}{I}$ and $(w,w') \in D_a$ such that $w \stackrel  {\overset{\multimapdotinv}{I}} {\asymp} y, w' \stackrel  {\overset{\multimapdotinv}{I}} {\asymp}  y'$.  Let $a_1 < \ldots < a_n$ be in $L(A)$ and $(U_1,V_1), \ldots ,(U_n,V_n)$ be open sets in $Y^2$ such that $(y(a_i),y'(a_i)) \in (U_i,V_i)$. To complete the proof, it suffices to show that there is $\beta < \alpha$ and $(z,z') \in E_{\beta}$ such that $(z(a_i),z'(a_i)) \in (U_i,V_i)$. By enlarging, if necessary, we assume that $a$ is one of the $a_i$'s. As $\alpha$ is a limit ordinal, we have that $A^{\beta} \rightarrow A^{\alpha}$ in the Hausdorff metric as $\beta \rightarrow \alpha$. Hence, we may choose $\beta <\alpha$ sufficiently large so that there is $J \in \C(A^{\beta})$ such
$I \cap \{a_1, \ldots, a_n\} \subseteq {J} \subseteq {I}$. Moreover, if $\ell (I) \in L(A)$, we may obtain $J$ so that, in addition, that $\ell(J) = \ell (I)$. Hence, for sufficiently large $\beta$ we have 
that ${\overset{\multimapdotinv}{I}} \cap \{a_1, \ldots, a_n\} \subseteq {\overset{\multimapdotinv}{J}} \subseteq {\overset{\multimapdotinv}{I}}$. We define $z, z'$ in the following fashion: $z = w$, $z' = w'$ on $L(A) \setminus \overset{\multimapdotinv}{J}$, $z(a_i)= y(a_i)$, $z'(a_i) = y'(a_i)$ for all $a_i$'s in $\overset{\multimapdotinv}{J}$ and $z(b) = w(b)$, $z'(b) = w'(b)$ for $b \in  \left (L(A) \cap \overset{\multimapdotinv}{J} \right ) \setminus \{a_1, \ldots, a_n\}$. Note that $w \stackrel  {\overset{\multimapdotinv}{J}} {\asymp} z, w' \stackrel  {\overset{\multimapdotinv}{J}} {\asymp}  z'$ and as $a \in \overset{\multimapdotinv}{I} \cap \{a_1, \ldots, a_n\} \subseteq \overset{\multimapdotinv}{J} $ and $(w,w') \in D_a$, we have that $(z,z') \in E_{\overset{\multimapdotinv}{J}} \subseteq E_{\beta}$. If $a_i \in \overset{\multimapdotinv}{I}$, then we have that $a_i \in \overset{\multimapdotinv}{J}$ and hence $z(a_i) = y(a_i)$ and $z'(a_i)=y'(a_i)$, implying that $(z(a_i),z'(a_i)) \in (U_i, V_i)$. If $a_i \notin \overset{\multimapdotinv}{I}$, then $a_i \notin \overset{\multimapdotinv}{J} $. In this case, we have that $z(a_i) = w(a_i) = y(a_i)$ as $w \stackrel  {\overset{\multimapdotinv}{I}} {\asymp} y$ and similarly $z'(a_i) = y'(a_i)$, yielding that $(z(a_i), z'(a_i)) \in (U_i, V_i)$ and completing the proof.
\end{proof}
\begin{proposition}\label{notfull}
Let $A \in \Kcal$ and $\alpha$ an ordinal. If  $I, J \in \C (A^{\alpha})$ are disjoint with 
$L(A) \cap \overset{\multimapdotinv}{I} \neq \emptyset$ and $L(A) \cap \overset{\multimapdotinv}{J} \neq \emptyset$, then $E_{\alpha} \neq Y_A^2$. 
\end{proposition}
\begin{proof}
Let $a \in L(A) \cap \overset{\multimapdotinv}{I}$ and $b \in L(A) \cap \overset{\multimapdotinv}{J}$. Let $(y,y') \in E_{\alpha}$. By the definitions of $D_A$, the properties of $I$ of $J$ and definition of $E_{\alpha}$, we have that either $y(a) = y'(a)$ or $y(b)= y'(b)$, implying that $E_{\alpha} \neq Y_A^2$. 
\end{proof}
\begin{proposition}\label{decomposition}
Let $A \in \Kcal$ and $\alpha$ be a countable ordinal such that $A^{\alpha} \neq \emptyset$. Let $I \in \C(A^{\alpha+1})$ and let $G$ be the collection of $\overset{\multimapdotinv}{J}$ such that $J \in \C(A^{\alpha})$, $J \subseteq I$ and $\overset{\multimapdotinv}{J} \cap L(A) \neq \emptyset$. Then,  $G$ is a collection of pairwise disjoint intervals such that $\overset{\multimapdotinv}{I} \cap L(A) \subseteq \cup G$.
\end{proposition}
\begin{proof}
That $G$ is pairwise disjoint is clear as elements of $\C(A^{\alpha})$ are pairwise disjoint open intervals. Let $a \in \overset{\multimapdotinv}{I} \cap L(A)$. If $a \in I \setminus A^{\alpha}$, then $a$ is in some interval contiguous to $A^{\alpha}$ which is a subset of $I$ and we are done. Now suppose that $a \in A^{\alpha}$. Recall that by the fact that $a \in L(A)$, $a$ is the left endpoint of some open interval contiguous to $A$. That interval is a subset of some open interval $J$ contiguous to $A^{\alpha}$. As $a \in A^{\alpha}$,  the left endpoint of $J$ is $a$. As $a \in \overset{\multimapdotinv}{I}$, we have that $J \subseteq I$. Then $a \in \overset{\multimapdotinv}{J} \in G$, concluding the proof.
\end{proof}
\begin{proposition}\label{GammaEalpha}
Let $A \in \Kcal$ with $A^{\alpha} \neq \emptyset$. Then, $E_{\alpha +1} \subseteq \Gamma(E_{\alpha})$.  \end{proposition}
\begin{proof} Let $I \in \C(A^{\alpha +1})$.
It will suffice to show that $E_{\overset{\multimapdotinv}{I}} \subseteq \Gamma(E_{\alpha})$. Let  \[G=\{\overset{\multimapdotinv}{J}: J \subseteq I,  \overset{\multimapdotinv}{J} \cap L(A) \neq \emptyset, J \in \C(A^{\alpha})\}.\] 
We will apply Proposition~\ref{Enlarge} to $\overset{\multimapdotinv}{I}$, $G$ and $H_J = E_J$ to conclude $E_{\overset{\multimapdotinv}{I}} \subseteq \Gamma(E_{\alpha})$. We commence by verifying the hypotheses of Proposition~\ref{Enlarge}. By Proposition~\ref{decomposition}, we have that $G$ is a collection of pairwise disjoint sub-intervals of $\overset{\multimapdotinv}{I}$ with $\overset{\multimapdotinv}{I} \cap L(A) \subseteq \cup G$. Next choose $f$ according to Proposition~\ref{fixedhigherlevel} and  let $t=t_I$. Let us check that Properties 1-4 of Proposition~\ref{Enlarge} are satisfied. Property 1 follows simply from  the definition of $D_a$ and $E_J$. To see Property 2, let $(y,y') \in E_J$. By the definition of $\Box$, $(y,y')$ is obtained from some $D_a$, $a \in J \cap L(A)$. By Proposition~\ref{fixedhigherlevel}, we have that $y|_{\overset{\multimapdotinv}{I} \cap J[}= t_a|_{\overset{\multimapdotinv}{I} \cap J[}=f$.  Analogously, Proposition~\ref{entuallysame} implies that $y|_{\overset{\multimapdotinv}{I}[} =t$. To see Property 3, note that as $\overset{\multimapdotinv}{J} \cap L(A) \neq \emptyset$, some $D_a$ is a subset of $E_J$. Property 4 is clear by the definition of $E_J$. Finally, set $H= \bigcup _{J \in G} E_J$.  By Proposition~\ref{closedE} and the fact that $\Delta_{Y_A} \subseteq E_{\alpha}$, we have that 
\[\overline{H} \subseteq \bigcup _{J \in G} E_J \bigcup \Delta_{Y_A} \subseteq  E_{\alpha}.\]
Now drawing on the conclusion of Proposition~\ref{Enlarge}, we have that 
\[ \Gamma (E_{\alpha}) \supseteq \Gamma (\overline{H}) \supseteq \stackrel {I}{\Box}H.
\]
As $\overset{\multimapdotinv}{I} \cap L(A) \subseteq \bigcup G$, we have that $\bigcup_{a \in \overset{\multimapdotinv}{I} \cap L(A) } D_a \subseteq H $, implying that  $\stackrel {\overset{\multimapdotinv}{I}}{\Box}H \supseteq E_{\overset{\multimapdotinv}{I}}$ and concluding the proof of $E_{\alpha +1} \subseteq \Gamma(E_{\alpha})$.
\end{proof}

\begin{proposition}\label{alphaEnlarge}
Let $A \in \Kcal$ with $A^{\alpha} \neq \emptyset$.  Then, 
\[ E_{\alpha} \subseteq \Gamma ^{\alpha} (D_A). 
\]
\end{proposition}
\begin{proof}
We proceed by induction on $\alpha$. 
At stage $\alpha =0$, we have that 
\[
E_0 = \bigcup_{I \in \C(A)} \stackrel {\overset{\multimapdotinv}{I}}{\Box}\bigcup _{a \in \overset{\multimapdotinv}{I} \cap L(A)}D_A= \bigcup_{a \in L(A)} \stackrel {a}{\Box}  D_a=D_A = \Gamma ^{0}(D_A).
\]
Suppose that $\alpha >0$ is a countable ordinal such that for all $\beta < \alpha$ we have that $E_{\beta} \subseteq \Gamma ^{\beta} (D_A)$. If $\alpha$ is a limit ordinal, then
\[ \Gamma^{\alpha}(D_A) = \overline {\bigcup _{\beta < \alpha} \Gamma^{\beta}(D_A )} \supseteq \overline {\bigcup _{\beta < \alpha} E_{\beta}}= E_{\alpha},
\]
where the rightmost equality follows by Proposition~\ref{basicprofEalpha}. If $\alpha$ is successor ordinal, then, by the induction hypothesis and  Proposition~\ref{GammaEalpha},
\[ \Gamma^{\alpha}(D_A) = \Gamma (\Gamma^{\alpha-1}(D_A)) \supseteq \Gamma (E_{\alpha-1}) \supseteq E_{\alpha},
\]
\end{proof}

Next we would like to study the situation when $A^{\infty} \neq \emptyset$. It may seem that in this situation, $\Gamma (E_{\alpha}) = E_{\alpha}$
where $\alpha = |A|_{CB}$. Unfortunately, such is not the case. We modify $E_{\alpha}$ appropriately to obtain a set $E$ for which $\Gamma (E) = E$.
\begin{proposition}\label{perfectstable}
Let $A \in \Kcal$ with $A^{\infty} \neq \emptyset$, $\alpha = |A|_{CB}$ and $I_e\in \C(A^{\alpha})$ such that $r(I_e) =1$. We set 
\[ E = \bigcup_{I \in \C(A^{\alpha})} \stackrel {\overset{\multimapdotinv}{I_e}}{\Box} E_{\overset{\multimapdotinv}{I}}.
\]
Then, $D_A \subseteq E$, $E$ is closed and $\Gamma(E) = E$. Moreover, $E \neq Y_A^2$.
\end{proposition}
\begin{proof}
That $D_A \subseteq E$ follows from the fact that $D_A=E_0 \subseteq E_{\alpha} \subseteq E$.  As $\Box$ is distributive over unions, Remark \ref{rem:boxbasic}, we have that 
\[ E = \stackrel {\overset{\multimapdotinv}{I_e}}{\Box} \bigcup_{I \in \C(A^{\alpha})}  E_{\overset{\multimapdotinv}{I}} = \stackrel {\overset{\multimapdotinv}{I_e}}{\Box} E_{\alpha}.
\]
As $E_{\alpha}$ is closed and $\Box$ of a closed set is closed, we have that $E $ is closed.

To prove that $\Gamma (E) = E$, first we will show that $E^+ = E$. 
  By Proposition~\ref{EI+}, $ \left (\stackrel{I_e}{\Box} E_{\overset{\multimapdotinv}{I}} \right) ^+ = \stackrel{I_e}{\Box} E_{\overset{\multimapdotinv}{I}}$ for all $I \in \C (A^{\infty})$.   As such, it will suffice to show that if $I, J \in \C (A^{\infty})$  are distinct, $(y,y') \in \stackrel{I_e}{\Box} E_{\overset{\multimapdotinv}{I}}$, and $(y',y'') \in \stackrel{I_e}{\Box} E_{\overset{\multimapdotinv}{J}}$, then $(y,y'') \in E$. Without loss of generality, assume that $I$ is to the left of $J$.  Using the definition of $E_{\overset{\multimapdotinv}{I}}$, and Remark \ref{rem:boxbasic}, there exists $a \in \overset{\multimapdotinv}{I}$ and $(w_1,w_2) \in D_a$ such that  $ w_1 \stackrel  {\overset{\multimapdotinv}{I} \cup \overset{\multimapdotinv}{I_e}} {\asymp} y$ and $ w_2 \stackrel  {\overset{\multimapdotinv}{I} \cup \overset{\multimapdotinv}{I_e}} {\asymp}  y'$. Similarly, let $b \in \overset{\multimapdotinv}{J}$ and $(w_3,w_4) \in D_b$ such that  $ w_3 \stackrel  {\overset{\multimapdotinv}{I} \cup \overset{\multimapdotinv}{I_e}} {\asymp} y '$ and $ w_4 \stackrel  {\overset{\multimapdotinv}{I} \cup \overset{\multimapdotinv}{I_e}} {\asymp}  y''$. We now show that $J$ must equal to $I_e$. Suppose such is not the case. As $A^{\infty}$ is perfect we have that between any two elements of $\C(A^{\infty})$ there is an element of $\C(A^{\infty})$. Moreover, every element of $\C(A^{\infty})$ whose endpoint is not $0$ contains a point of $L(A)$. Hence, we may choose some $c \in L(A) \cap  J[ $ such that $ c \notin \overset{\multimapdotinv}{I_e}$. 
By Proposition~\ref{perfectbarrier}, the definition of $D_A$ and the fact that $(w_1,w_2) \in D_a$, we have that 
 $ w_1(c) = w_2(c) =r (I)$. As $ w_1 \stackrel  {\overset{\multimapdotinv}{I} \cup \overset{\multimapdotinv}{I_e}} {\asymp} y$,  $ w_2 \stackrel  {\overset{\multimapdotinv}{I} \cup \overset{\multimapdotinv}{I_e}} {\asymp}  y'$, $c \in L(A) \in J[ $ and $ c \notin ]\overset{\multimapdotinv}{I_e}$, we have that  $ y(c) = y'(c) =r (I)$. Arguing similarly with $(w_3,w_4)$, we have that $ y'(c) = y''(c) =r (J)$, yielding a contradiction as $I$ and $J$ are  disjoint intervals. Hence we have shown that $J =I_e$. We now show that $(y,y'') \in \stackrel{I_e}{\Box} E_{\overset{\multimapdotinv}{I}}$ to complete the proof of $\Gamma (E) = E$. It will suffice to show that $w_2 \stackrel  {\overset{\multimapdotinv}{I} \cup \overset{\multimapdotinv}{I_e}} {\asymp} y''$ as already $ w_1 \stackrel  {\overset{\multimapdotinv}{I} \cup \overset{\multimapdotinv}{I_e}} {\asymp} y$.
 Indeed, $y'' = y'$ on $L(A) \setminus \overset{\multimapdotinv}{I_e}$, and $y' = w_2$ on $L(A) \setminus (\overset{\multimapdotinv}{I} \cup \overset{\multimapdotinv}{I_e})$, implying that $y'' = w_2$ on $L(A) \setminus (\overset{\multimapdotinv}{I} \cup \overset{\multimapdotinv}{I_e})$, i.e., $w_2 \stackrel  {\overset{\multimapdotinv}{I} \cup \overset{\multimapdotinv}{I_e}} {\asymp} y''$.
 
 Finally, we show that $E \neq Y_A^2$ to complete the proof of the proposition. Choose $I, J \in \C(A^{\infty})$ with neither one being $I_e$ and $a \in L(A) \cap \overset{\multimapdotinv}{I} $ and $b \in L(A) \cap \overset{\multimapdotinv}{J}$. Let $(y,y') \in E$. By the definitions of $D_A$ and $E$, we have that either $y(a) = y'(a)$ or $y(b)= y'(b)$, implying that $E \neq Y_A^2$.

\end{proof}

\begin{corollary}
\label{cor:Acountable}
A set $A \in \Kcal$ is countable if and only if $\Gamma^{\infty}(D_A) = Y_A ^2$.
\end{corollary}
\begin{proof}
Let $\beta = | A |_{CB}$. 

If $A$ is countable, then $\beta$ is necessarily a successor ordinal and $A^{\beta} = \emptyset$. This implies that \[
E_{\beta}=E_{(0,1)}=\stackrel {(0,1)}{\Box} \bigcup _{a \in (0,1) \cap L(A)} D_a=Y_A^2. \] 

Applying  Proposition~\ref{alphaEnlarge} to $\alpha = \beta -1$, we have that $E_{\beta} \subseteq \Gamma^{\alpha}(D_A)$. Hence, $\Gamma^{\infty}(D_A) = Y_A ^2$.  

If $A$ is uncountable, then $\Gamma^{\infty}(D_A) \neq Y_A ^2$ follows simply from  Proposition~\ref{perfectstable}.
\end{proof}
\begin{proposition}\label{highrank}
For each countable ordinal $\alpha$, there exists $A \in \K$ so that the $\Gamma$-rank of $B_A$ is greater than $\alpha$.
\end{proposition}
\begin{proof}  Let $\alpha$ be a countable ordinal. Let $A \in \K$ be a countable compact set such that $|A|_{CB} > \alpha+1$ and $A^{|A|_{CB}-1} = \{\max(A)\}$. For the sake of brevity, let $e = \max(A)$ and  $I_e \in \C(A)$ with $\ell(I_e) = e$. By our choices of $A$ we have that $I_e \in \C(A^{\beta})$ for all $\beta < |A|_{CB}$.

Let us first prove that $\Gamma(D_A) = \Gamma (E_0)= \stackrel{e}{\Box} E_{1}.$
For each $J \in \C(A)$, we have that \begin{equation}
   \left ( E_{\overset{\multimapdotinv}{J}} \cup E_{\overset{\multimapdotinv}{I_e}} \right )^+= \stackrel{e}{\Box}E_{\overset{\multimapdotinv}{J}}.
\end{equation} 
Next, if $J_1, J_2 \in \C(A)$ with $J_i \subseteq I_i \in \C(A')$ with neither $I_i$ equal to $I_e$, and $I_1 \neq I_2$, then we have that 
\begin{equation}
 \left ( \stackrel{e}{\Box}E_{\overset{\multimapdotinv}{J_1}}  \cup  \stackrel{e}{\Box}E_{\overset{\multimapdotinv}{J_2}} \right )^+ = \stackrel{e}{\Box}E_{\overset{\multimapdotinv}{J_1}}  \cup  \stackrel{e}{\Box}E_{\overset{\multimapdotinv}{J_2}}  .  
\end{equation} 
This is so because there is $I_3 \in \C(A')$ which is to the right of $I_1,I_2$ and to the left of $I_e$.  For $a \in L(A) \cap I_3$, we have that $t_{1_1}(a) \neq t_{I_2}(a)$. Hence (2) follows. Next, fix
$I \in \C(A')$ momentarily. Arguing as in the proof of Proposition~\ref{Enlarge} with $H_J = \stackrel{e}{\Box}E_{\overset{\multimapdotinv}{J}}$,  $J \in \C(A)$, $J \subseteq I$, and using (2) from above, we have that
\begin{equation}
   \overline {\left ( \bigcup _{J \in \C(A), J \subseteq I} \stackrel{e}{\Box}E_{\overset{\multimapdotinv}{J}} \right )^+} = \stackrel{e}{\Box}E_{\overset{\multimapdotinv}{I}}.  
\end{equation}
Finally, putting (1), (2) and (3) together, we have that 
\begin{align*}
    \Gamma (E_0) &= \overline{E_0^+} = \overline{ \left (  \bigcup_{J \in \C(A)}  E_{\overset{\multimapdotinv}{J}} \right ) ^+ } = \overline{ \left (  \bigcup_{J \in \C(A)}  E_{\overset{\multimapdotinv}{J}} \cup E_{\overset{\multimapdotinv}{I_e}} \right ) ^+ } = \overline{ \left (  \bigcup_{J \in \C(A)}  \left ( E_{\overset{\multimapdotinv}{J}} \cup E_{\overset{\multimapdotinv}{I_e}} \right )^+ \right ) ^+ } \\ & \stackrel{(1)}= \overline{ \left (  \bigcup_{J \in \C(A)}  \stackrel{e}{\Box}E_{\overset{\multimapdotinv}{J}} \right ) ^+ } = \overline  { \left (\bigcup_{I \in \C(A')} \left ( \bigcup _{J \in \C(A), J \subseteq I} \stackrel{e}{\Box}E_{\overset{\multimapdotinv}{J}} \right )^+ \right )^+} \\ &
    \stackrel{(2)}= \overline  { \bigcup_{I \in \C(A')} \left ( \bigcup _{J \in \C(A), J \subseteq I} \stackrel{e}{\Box}E_{\overset{\multimapdotinv}{J}} \right )^+} 
    = \overline  { \bigcup_{I \in \C(A')} \overline{\left ( \bigcup _{J \in \C(A), J \subseteq I} \stackrel{e}{\Box}E_{\overset{\multimapdotinv}{J}} \right )^+ }} \\
    &\stackrel{(3)}= \overline{ \bigcup_{I \in \C(A')} \stackrel{e}{\Box}E_{\overset{\multimapdotinv}{I}}} = \stackrel{e}{\Box} E_{1}.
\end{align*}
The last equality follows from the fact that $E_1$ is closed.

The general $\alpha$ case is analogous to the above.
Arguing in the same fashion, we have that  $\Gamma ^{\beta}  (D_A)=\Gamma ^{\beta}  (E_0) = \stackrel{e}{\Box} E_{\beta + 1}$, for all $\beta < |A|_{CB}$. As $A^{|A|_{CB}-2}$ is an infinite set, we have that $\stackrel{e}{\Box} E_{|A|_{CB}-2} \neq Y_A^2$, implying that the $\Gamma$-rank of $D_A$, and hence of $B_A$, is greater than or equal to $|A|_{CB}-1 > \alpha$.
\end{proof}

\subsection{Entropy Construction}\label{sec:2}

This section is less technical than the preivous one. The main objective is to translate Proposition \ref{highrank} to the language of entropy pairs, and check that the phase space is indeed a Cantor space.

In this subsection we will map $(Y_A,R_A, B_A)$ from the previous section to a dynamical system $(X_A, T_A)$ so that $\Gamma ^{\infty}  (B_A) = Y_A^2$ if and only $\Gamma ^{\infty}  ((E(X_A,T_A)) = X_A^2$.

Fix $A \in \K$ and let  $(Y_A,R_A,B_A)$ be as defined in the last section. We continue to use the notation of the previous section. Recall that for $a\in L(A)$, $I \in \C(A)$ with $a= \ell (I)$,

\[B_a =\{ (y, y')\in  Y_A^2:  y \stackrel  {a} {\asymp} y', \{y(a),y'(a)\} \in R \ \ 
 \& \ \ y|_{I[} = t_a|_{I[} \}.
\]

\[B_A =  \bigcup_{a  \in L(A)}  (B_a  \cup \Delta _{Y_A}). \]

We will construct $X_A$ as a factor. 
We let $Z_{A}=X\times Y_{A}$ be our intermediate space. Points in $Z_{A}$ are represented as $xy$, with $x\in X$ and $y\in Y_A$. For $a\in L(A)$, $I \in \C(A)$ with $a= \ell (I)$, we let 
\begin{align*}
F_{a}  & = \left \{(xy,xy^{\prime})\in Z_{A}^{2}:(y,y^{\prime})\in
B_{a}\text{ and }\{y(a),y^{\prime}(a)\}=\{x,T(x)\} \right \},\text{ and}
\end{align*}
\[F_{A}=\bigcup_{a \in L(A)}(F_{a} \cup \Delta_{Z_{A}}).\]
 \begin{lemma}\label{relationshipBnF}
 If $(y,y') \in B_A$, then there is $x \in X $ such that $(xy,xy') \in F_A$.
 \end{lemma}
\begin{proof}
Let $(y,y') \in B_A$. If $y=y'$, then for any $x \in X$ we have that $(xy,xy') \in F_A$ as $\Delta_{Z_{A}} \subseteq F_A$. If $y \neq y'$, then $(y,y') \in B_a$ for some $a \in L(A)$ and $\{y(a),y'(a)\} \in R$. Without loss of generality, assume that $R(y(a)) = y'(a)$. Then, letting $x = y(a)$ we have that $(xy,xy') \in F_a \subseteq F_A$. 
\end{proof}
\begin{lemma}
  $F_{A}$ is an equivalence relation  which is a closed subset of $Z_A^2$. Moreover, $F_A$ is $T\times R_{A}$-invariant.
\end{lemma}
\begin{proof}
That $F_A$ is a closed subset of $Z_A^2$ follows from the fact that $B_A$ is a closed subset of $Y_A ^2$ (Proposition \ref{closed}). $B_A$ is $R_A$-invariant as the range of $t_a$, $a \in L(A)$, is a subset of $C \subseteq fix(Y,R)$.  This implies that $F_A$ is  $T\times R_{A}$-invariant. That $\Delta_{Z_{A}}\subseteq F_{A}$ follows from the definition of $F_A$. As $B_A$ is symmetric, so is $F_{A}$.

Finally, we show that $F_A$ is transitive, i.e., $F_A^+ = F_A$. Let $(xy,xy^{\prime}),(xy^{\prime},x
y^{\prime\prime})\in F_{A}$. If $y = y'$ or $y' =y''$, then obviously 
$(xy,xy^{\prime\prime})\in F_{A}$. Hence, assume that $y\neq y^{\prime}\neq y^{\prime\prime}$.
Let $a, a^{\prime}\in
L(A)$ the only points where $y(a)\neq y^{\prime}(a)$ and $y^{\prime}
(a^{\prime})\neq y^{\prime\prime}(a^{\prime})$. By hypothesis we have that
$\{y(a),y^{\prime}(a)\}=\{y^{\prime}(a^{\prime}),y^{\prime\prime}(a^{\prime
})\}=\{                      x,Tx\}$. Since $y\neq y'$ then $Tx\neq x$. We next show that $a = a'$. To obtain a contradiction, assume that $a<a'$. Since $(y,y')\in B_a$ then $y'(a') = t_a(a')$ which by construction is a fixed point of $R$, yielding a contradiction. Hence,
$a=a^{\prime}$. Using $a=a^{\prime}$ and that $\{y(a),y^{\prime}(a)\}=\{y^{\prime}(a^{\prime}),y^{\prime\prime}(a^{\prime
})\}=\{                      x,Tx\}$, we obtain $y(a)=y''(a)$, implying that $y =y''$.  As $\Delta_{Z_{A}}\subseteq F_{A}$, we have that $(xy,xy'') =(xy,xy) \in F_A$, completing the proof.
\end{proof}

Let $X_{A}=Z_{A}/F_{A}$ endowed with the quotient topology. As $Z_{A}$ is a
compact metrizable space and $F_{A}$ a closed set and an equivalence relation, we
have that $X_{A}$ is a compact metrizable space \cite[Theorem 4.2.13]%
{engleking1989general}. Let $P_{A}:Z_{A}\rightarrow X_{A}$ be projection
induced by $F_{A}$ and define $T_{A}:X_{A}\rightarrow X_{A}$ by $T_{A}
([z]_{F_{A}})=[(T\times R_{A})(z)]_{F_{A}}$, where $[z]_{F_{A}}$ denotes the
equivalence class of $z\in Z_{A}$ in $F_{A}$. As $F_{A}$ is invariant under
$T\times R_{A}$, we have that $T_{A}$ is well-defined and continuous.
Moreover, we have that $P_{A}\circ(T\times R_{A})=T_{A}\circ P_{A}$. In
particular, $(X_{A},T_{A})$ is a factor $(Z_{A},T\times R_{A})$, i.e., the
following diagram commutes.
\[
\begin{tikzcd}
	Z_{A}\arrow[r, "T\times R_{A}"] \arrow[d,"P_{A}" ]
	&Z_{A}\arrow[d, "P_{A}" ] \\
	X_{A} \arrow[r,  "T_{A}" ]
	& X_{A} \end{tikzcd}
\]
As the product of mixing maps is mixing, we have that $T\times R_{A}$ is
mixing. As $P_{A}$ is surjective, and mixing is preserved under factor maps, we
have that $(X_{A},T_{A})$ is a mixing as well.





\begin{lemma}
\label{lem:E2}
Let $A\in \Kcal$. We have that
\[
E(X_{A},T_{A})\subseteq\{([xy]_{F_{A}},[x^{\prime}y^{\prime
}]_{F_{A}}):x,x^{\prime}\in X, (y,y^{\prime})\in B_{A}\}\subseteq E(X_{A},T_{A})^+
\]

\end{lemma}

\begin{proof}
Recall from Lemma~\ref{startup} that 
 $E(Y_{A},R_{A})=\Delta_{Y_{A}}$ and  $E(X,T)=X^{2}$. Using these facts and  Theorem \ref{thm:basic} (2) we obtain that
\[
E(Z_{A},T\times R_{A})=\{(xy,x^{\prime}y):x,x^{\prime}\in X,y\in
Y_{A}\}.
\]
Using Theorem \ref{thm:basic} (3) we get
\[
E(X_{A},T_{A})=\{([xy]_{F_{A}},[x^{\prime}y]_{F_{A}}):x,x'\in X,y\in
Y_{A}\}. 
\]
It follows that
\[
E(X_{A},T_{A})\subseteq\{([xy]_{F_{A}},[x^{\prime}y^{\prime
}]_{F_{A}}):x,x^{\prime}\in X,\text{ and }(y,y^{\prime})\in B_{A}\}.
\]
To prove the second $\subseteq$, let $x, x' \in X$ and $(y,y')\in B_A$. Then, by Lemma~\ref{relationshipBnF}, there exists $x''\in X$ such that 
\[
[x''y]_{F_A}=[x''y']_{F_A}.
\]
Moreover, we have that \[
([xy]_{F_A},[x''y]_{F_A}),([x''y']_{F_A},[x'y']_{F_A})\in E(X_{A},T_{A}),
\]
implying that $([xy]_{F_A},[x'y']_{F_A}) \in E(X,T)^+$ and concluding the proof.
\end{proof}

\begin{theorem}
\label{thm:BtoX}
Let $A\in \Kcal$. Then, the entropy rank of $(X_A,T_A)$ is greater than or equal to the $\Gamma$-rank of $B_A$, and 
\[\Gamma^{\infty}(E(X_A,T_A)) = \{([xy]_{F_{A}},[x'y']_{F_{A}}):x,x'\in X, (y,y'
)\in\Gamma^{\infty}(B_{A})\}
.\] 
In particular, we have that $\Gamma ^{\infty}(B_A)=Y_A^2$ if and only $\Gamma ^{\infty}((E(X_A,T_A)) = X_A^2$.
\end{theorem}
\begin{proof}
Let
\begin{equation*}
 G=\{([xy]_{F_{A}},[x'y']_{F_{A}}):x,x'\in X,(y,y'
)\in B_{A}\}.  
\end{equation*}
By Lemma \ref{lem:E2} we have that 
\[
\Gamma^{\alpha} (E(X_{A},T_{A}) ) \subseteq \Gamma^{\alpha} (G) \}\subseteq \Gamma^{\alpha+1} (E(X_{A},T_{A})),
\]
for every countable ordinal $\alpha$. We next prove by induction   that
\begin{equation*}\label{eqngalpha}
   \Gamma^{\beta}(G)
=\{([xy]_{F_{A}},[x'y']_{F_{A}}):x,x'\in X,\text{ and }(y,y'
)\in\Gamma^{\beta}(B_{A})\}. \tag{*}
\end{equation*}
Assume the result holds for
every $\beta^{\prime}<\beta$. If $\beta$ is a successor, then it is easy to check
\[
(\Gamma^{\beta-1}(G))^+= \left \{([xy]_{F_{A}},[x^{\prime
}y^{\prime}]_{F_{A}}):x,x^{\prime}\in X,\text{ and }(y,y^{\prime}
)\in \left (\Gamma^{\beta-1}(B_{A}) \right )^+ \right \}.
\]
If $\beta$ is limit ordinal, one can show that
\[
 \bigcup _{\beta'<\beta} \Gamma^{\beta'}(E(X_{A},T_{A}))= \left  \{([xy]_{F_{A}},[x^{\prime
}y^{\prime}]_{F_{A}}):x,x^{\prime}\in X,\text{ and }(y,y^{\prime}
)\in\bigcup _{\beta'<\beta} \Gamma^{\beta'}(B_{A}) \right \}.
\]
Now using the continuity of $P_A$, we have that (\ref{eqngalpha}) holds.

Returning to the main proof, 
by (\ref{eqngalpha}) we have that the $\Gamma$-rank of $G$ is the same as the $\Gamma$-rank of $B_A$. As $E(X_A,T_A) \subseteq G$, we have that the $\Gamma$-rank of $G$ is less than or equal to the entropy rank of $(X_A,T_A)$, or equivalently, the entropy rank of $(X_A,T_A)$ is greater than or equal to $\Gamma$-rank of $B_A$. Finally, the second claim of the lemma follows from the fact that,
\[
\Gamma^{\alpha} (E(X_{A},T_{A}) ) \subseteq \left \{([xy]_{F_{A}},[x^{\prime}y^{\prime
}]_{F_{A}}):x,x^{\prime}\in X,\text{ and }(y,y^{\prime})\in \Gamma^{\alpha} (B_{A}) \right \}\subseteq \Gamma^{\alpha+1} (E(X_{A},T_{A})),
\]
for every countable ordinal $\alpha$. 
\end{proof}

\begin{corollary}\label{highcperank}
Let $\alpha$ be a countable ordinal. Then, there is a countable set $A \in \K$ such that $(X_A,T_A)$ has CPE and its entropy rank is greater than $\alpha$.
\end{corollary}
\begin{proof}
Let $A \in \K$ be a countable set in the statement of  Proposition~\ref{highrank}. By Corollary~\ref{cor:Acountable} and Remark~\ref{ignoreB}, we have that $\Gamma^{\infty} (B_A) = Y_A^2$. Now by Lemma~\ref{thm:BtoX}, we have that the entropy rank of $(X_A, T_A)$ is greater than $\alpha$ and $\Gamma ^{\infty}((E(X_A,T_A)) = X_A^2$, i.e., $(X_A,T_A)$ has CPE.
\end{proof}
\begin{theorem}\label{CantorSpace}
For each $A\in \Kcal$, $X_{A}$ is a Cantor space.
\end{theorem}
\begin{proof}
We will show that $X_A$ has no isolated points and $X_A$ is totally disconnected. 
To see that $X_A$ is totally disconnected, let  $([xy]_{F_{A}},[x'y^{\prime}]_{F_{A}})\in
X_{A}^{2}\backslash\Delta_{X_{A}}$. We will show that there exists a clopen set of $X_A$ which contains one but not the other.

Recall that $P_{A}:Z_{A}\rightarrow X_{A}$ is a surjection and $X_A$ is equipped with quotient topology. Hence, if $D \subseteq X_A$ and  $P_{A}^{-1}(D)$ is 
clopen in $Z_A$, then $D$ is clopen in $X_A$.

\textbf{Case 1:} Suppose that $x\neq x'$. As $X$ is zero-dimensional, we may choose a clopen set  $B\subseteq X$ such that
$x\in B$ and $x'\notin B$. Then, $D=B\times Y_{A}$ is a clopen set.
One can check that $P_{A}^{-1}([D]_{F_{A}})=D$. Hence, $[D]_{F_{A}}$ is a
clopen set such that $[xy]_{F_{A}}\in [D]_{F_{A}}$ and $[x^{\prime
}y^{\prime}]_{F_{A}}\notin [ D]_{F_{A}}$.

 \textbf{Case 2:} Suppose that $x = x'$. Since $[xy]_{F_{A}}\neq\lbrack xy^{\prime
}]_{F_{A}}$, without loss of generality, we may assume that there exists $b\in L(A)$ such
that $y(b)\neq y'(b)$ and $y(b)\notin\left\{x,T(x)\right\}.$ Let $B\subseteq X$
be a clopen set such that $y(b)\notin B$, $\{  x,T(x),y'(b)\}\subseteq B$, and using the continuity of $T$ let $F\subseteq B$  be a clopen set such that $x\in F$, $F\cup T(F)\subseteq B$.

We define 
\[
K_{a}=\left\{
\begin{array}
[c]{cc}
Y & \text{if }a\neq b\\
B^c\cap Y & \text{if }a=b
\end{array}
\right.
\]
and
\[
K=F\times \prod_{a\in L(A)}K_a\subseteq Z_A.
\]
Observe that for $k\in K$ and $k' \in [k]_{F_A}$, we have that $k(b) = k'(b)$. From this it follows that  $[k]_{F_A}\subseteq K$. 
Hence, $P_{A}^{-1}([ K]_{F_{A}\text{ }})= K$, implying that $[ K]_{F_{A}}$ is clopen. Furthermore, $[xy]_{F_{A}}
\in\lbrack  K]_{F_{A}}$ and $[xy^{\prime}]_{F_{A}}\notin\lbrack
 K]_{F_{A}}$. 
 
With these two cases we conclude that $X_A$ is totally disconnected.
 
 To see that $X_A$ has no isolated points, note that $X$ and hence $Z_A$ has no isolated points. Moreover, $P_A$ is at most two-to-one map, implying that $X_A$ has no isolated point. 
 \end{proof}
\begin{theorem}\label{mainmixincpe} Let $X$ be a Cantor space. Then, $\textup{Mix}(X) \cap \textup{CPE}(X)$ is a coanalytic subset that is not Borel. 
\end{theorem}
\begin{proof} That $\textup{Mix}(X) \cap \textup{CPE}(X)$ is coanalytic follows from the facts that  $\textup{Mix}(X)$ is Borel (Remark \ref{rem:mix}), and $\textup{CPE}(X)$ is coanalytic  (Proposition~\ref{cpeiscoanalytic}). By Proposition~\ref{rankpi11}, we have that entropy rank is a $\Pi^1_1$-rank. By Corollary~\ref{highcperank} and Theorem~\ref{CantorSpace}, we have that for each countable ordinal $\alpha$ there is a TDS on $X$ whose entropy rank is bigger than $\alpha$. Now from Theorem~\ref{overspill}, we have that $\textup{Mix (X)} \cap \textup{CPE}(X)$ is not Borel.
\end{proof}

\begin{remark}
\label{rem:densep copy(1)} It is not difficult to check that if $(X,T)$ and $(Y,R)$
have a dense set of  periodic points, then $(X_{A},T_{A})$ also has a dense set of periodic points. This implies that the family of TDSs on a Cantor space with Devaney chaos and having CPE is not Borel. 
\end{remark}

\section{Questions}
We finish the paper with some open problems.

In Section~\ref{sec:mixingNonBorel}, we proved that $\textup{Mix(X)} \cap \textup{CPE(X)}$ is coanalytic and not Borel. Although we constructed a correspondence whose domain is a complete coanalytic set, it seems that the  correspondence is not Borel. Hence, we could not conclude that $\textup{Mix(X)} \cap \textup{CPE(X)}$ is complete coanalytic. Assuming $\Sigma^1_1$-Determinacy,  every coanalytic non-Borel set is complete coanalytic \cite[Theorem 26.8]{Kechris}. However, such is not the case in ZFC, begging the following question.
\begin{question}
Is there a ZFC proof of the fact that $\textup{Mix(X)} \cap \textup{CPE(X)}$ is complete coanalytic?
\end{question}

In Theorem \ref{intcomplete} we proved that the family of interval maps that have CPE is complete coanalytic. It would interesting to determine if one can prove a similar result about smooth maps, where smooth could mean differentiable, $C^1$, etc.
\begin{question}
What is the descriptive complexity of  the family of smooth maps, defined on some manifold, having CPE?
\end{question}
An answer would likely depend on the topology on the space of smooth maps. 

The construction of  minimal TDS that have CPE but not UPE is, in general, complicated and technical \cite{song2009minimal}. It is not known if one can construct such maps with arbitrarily high ranks.
\begin{question}
Given a countable ordinal $\alpha$ is there a minimal TDS having CPE with entropy rank greater than $\alpha$? 
\end{question}
As we stated in Theorem~\ref{thm:blokh}, for mixing graph maps properties of UPE and CPE coincide. The following questions asks what happens in higher dimensions.
\begin{question}
Do CPE and UPE coincide for mixing maps on a manifold of dimension $2$ or greater? If not, what is their descriptive complexity?
\end{question}
The final question concerns $\textup{UPE(X)}$.
\begin{question}
In Corollary~\ref{cor:upeborel}, we proved that $\textup{UPE(X)}$ is Borel. What is its exact descriptive complexity? How does it vary as we change the topological space $X$?
\end{question}
For example, there are compact connected subsets of ${\mathbb R}^3$ which only admit identity and constant maps \cite{hcook}. For such a space, the collection of UPE is empty.

 
\bibliographystyle{plain}
\bibliography{references}

\noindent Udayan B. Darji, ubdarj01@louisville.edu,\\
{\em Department of Mathematics, University of Louisville,
Louisville,
KY 40292, USA.}

\noindent Felipe Garc\'ia-Ramos,  fgramos@conacyt.mx,\\
{\em  Physics Institute, Universidad Aut\'onoma de San Luis Potos\'i, Mexico}\\
{\em  Faculty of Mathematics and Computer Science, Jagiellonian University, Poland.}\\

\end{document}